\documentclass[11pt]{amsart}

\usepackage{epigamath}

\usepackage[english]{babel}

\numberwithin{equation}{section}

\usepackage[all,cmtip]{xy} 

\usepackage{mathtools}
\usepackage{amsmath, amssymb, amsfonts, latexsym, mdwlist, amsthm}
\usepackage{subfig}
\usepackage{graphicx}
\usepackage{wrapfig}

\usepackage{tikz}
\usetikzlibrary{calc,trees,positioning,arrows,chains,shapes.geometric,%
	decorations.pathreplacing,decorations.pathmorphing,shapes,%
	matrix,shapes.symbols}

\tikzset{
	>=stealth',
	punktchain/.style={
		rectangle,
		rounded corners,
		draw=black, thick,
		minimum height=3em,
		text centered,
		on chain},
	line/.style={draw, thick, <-},
	element/.style={
		tape,
		top color=white,
		bottom color=blue!50!black!60!,
		minimum width=8em,
		draw=blue!40!black!90, very thick,
		text width=10em,
		minimum height=3.5em,
		text centered,
		on chain},
	every join/.style={->, thick,shorten >=1pt},
	decoration={brace},
	tuborg/.style={decorate},
	tubnode/.style={midway, right=2pt},
}

\usepackage{paralist}
\usepackage[shortlabels]{enumitem} 
\usepackage{graphicx}

\usetikzlibrary{patterns}

\setlist[enumerate,1]{label={\rm(\alph*)}, ref={\rm\alph*}}

\newtheorem{Thm}{Theorem}[section]
\newtheorem{Prop}[Thm]{Proposition}

\newtheorem{Lem}[Thm]{Lemma}

\newtheorem{Cor}[Thm]{Corollary}

\newtheorem{thm-int}{Theorem}

\theoremstyle{definition}
\newtheorem{Def-s}[Thm]{Definition}
\newtheorem{Def}[Thm]{Definition}

\theoremstyle{remark}
\newtheorem{Rem}[Thm]{Remark}

\renewcommand\_{^{}_}

\newcommand\PP{\mathbb P}
\newcommand\C{\mathbb C}
\newcommand\Q{\mathbb Q}
\newcommand\R{\mathbb R}

\newcommand\Z{\mathbb Z}
\newcommand\cA{\mathcal A}
\newcommand\cD{\mathcal D}
\newcommand\cH{\mathcal H}
\newcommand\cO{\mathcal O}
\newcommand\cI{\mathcal I}
\newcommand\cJ{\mathcal J}

\newcommand\cT{\mathcal T}
\newcommand\cS{\mathcal S}
\newcommand\cP{\mathcal P}
\newcommand\cF{\mathcal F}
\newcommand\ch{\operatorname{ch}}
\newcommand\ev{\operatorname{ev}}
\newcommand\Hom{\operatorname{Hom}}

\newcommand\rk{\operatorname{rank}}
\newcommand\ku{\operatorname{Ku}}

\def\abs#1{\left\lvert#1\right\rvert}

\newcommand\Coh{\operatorname{Coh}}
\def\Db{\mathop{\mathrm{D}^{\mathrm{b}}}\nolimits}
\def\Stab{\mathop{\mathrm{Stab}}\nolimits}
\def\Aut{\mathop{\mathrm{Aut}}\nolimits}

\newcommand\cB{\mathcal B}
\newcommand\cN{\mathcal N}

\makeatletter
\newtheorem*{rep@theorem}{\rep@title}
\newcommand{\newreptheorem}[2]{%
	\newenvironment{rep#1}[1]{%
		\def\rep@title{#2 \ref{##1}}%
		\begin{rep@theorem}}%
		{\end{rep@theorem}}}
\makeatother

\EpigaVolumeYear{7}{2023} \EpigaArticleNr{1} \ReceivedOn{May 25, 2022}
\InFinalFormOn{August 19, 2022}
\AcceptedOn{September 16, 2022}

\title{Serre-invariant stability conditions and \\ Ulrich bundles on cubic threefolds}
\titlemark{Serre-invariant stability conditions and Ulrich bundles}

\author{Soheyla Feyzbakhsh}
\address{Department of Mathematics, Imperial College, London, SW7 2AZ, United Kingdom}
\email{feyzbakhsh@imperial.ac.uk}

\author{Laura Pertusi}
\address{Dipartimento di Matematica F.\ Enriques, Universit\`a degli studi di Milano, Via Cesare Saldini 50, 20133 Milano, Italy}
\email{laura.pertusi@unimi.it}

\authormark{S.~Feyzbakhsh and L.~Pertusi}

\AbstractInEnglish{We prove a general criterion which ensures that a fractional Calabi--Yau category of dimension at most $2$ admits a unique Serre-invariant stability condition, up to the action of the universal cover of $\text{GL}^+_2(\R)$. We apply this result to the Kuznetsov component $\ku(X)$ of a cubic threefold $X$. In particular, we show that all the known stability conditions on $\ku(X)$ are invariant with respect to the action of the Serre functor and thus lie in the same orbit with respect to the action of the universal cover of $\text{GL}^+_2(\R)$. As an application, we show that the moduli space of Ulrich bundles of rank at least $2$ on $X$ is irreducible, answering a question asked by Lahoz, Macr\`{\i} and Stellari in~\cite{macri:acm-bundles-cubic-3fold}.}



\acknowledgement{S.F. is supported by EPSRC postdoctoral fellowship EP/T018658/1. L.P. is supported by the national research project PRIN 2017 Moduli and Lie Theory.}

\begin{document}



\maketitle

\begin{prelims}

\DisplayAbstractInEnglish

\bigskip

\DisplayKeyWords

\medskip

\DisplayMSCclass
 
\end{prelims}


\newpage

\setcounter{tocdepth}{1}

\tableofcontents


\section{Introduction}
Let $X$ be a smooth cubic hypersurface in the complex projective space $\PP^4$. As first noted by Kuznetsov~\cite{kuznetsov-derived-category-cubic-3folds-V14}, the birational geometry of $X$ is controlled by a certain full admissible subcategory $\ku(X)$ of the bounded derived category $\Db(X),$ defined by the semiorthogonal decomposition
$$\Db(X)= \langle \ku(X), \cO_X, \cO_X(H) \rangle.$$
We call $\ku(X)$ the Kuznetsov component of $X$. It has been shown that $\ku(X)$ completely determines the isomorphism class of the cubic threefold $X$; see~\cite{macri:categorical-invarinat-cubic-threefolds,feyz:desing}. 
The proof involves the theory of stability conditions for complexes in the derived category as introduced by Bridgeland~\cite{bridgeland:stability-condition-on-triangulated-category}, and the study of moduli spaces of stable objects in $\ku(X)$. 

Up to now there are two different constructions of stability conditions on $\ku(X)$: one in~\cite{macri:categorical-invarinat-cubic-threefolds} (see Section~\ref{sec-inducedstabcond}) and one in~\cite{bayer:stability-conditions-kuznetsov-component} (see Section~\ref{subsection-kuz-stability}). 
One of the goals of this paper originates from the desire to understand the connection between these two families of stability conditions.

\subsection{Main results}
Let $\cT$ be a $\C$-linear triangulated category with Serre functor $S$. Recall that the space of stability conditions on $\cT$ carries a right action of $\widetilde{\text{GL}}^+_2(\R)$, which is the universal cover of $\text{GL}^+_2(\R)$; 
see Section~\ref{sec-review}. The first result of this paper is a general criterion which implies the existence of a unique orbit of stability conditions which are invariant with respect to the action of the Serre functor (Definition~\ref{def-Sinvstab}).   

\begin{Thm}[Theorem~\ref{thm_uniqueSinv}] \label{thm_uniqueSinv_intro}
Let $\sigma_1$ and $\sigma_2$ be Serre-invariant stability conditions on a linear triangulated category $\cT$ satisfying the conditions~\eqref{C1},~\eqref{C2} and~\eqref{C3} in Section~\ref{section-serre-invariant}. Then there exists a $\tilde{g} \in \widetilde{\emph{GL}}^+(2, \mathbb{R})$ such that $\sigma_1= \sigma_2 \cdot \tilde{g}$. \end{Thm}

Theorem~\ref{thm_uniqueSinv_intro} applies for instance to the Kuznetsov component of a cubic threefold $X$ 
to show that the known stability conditions on $\ku(X)$ are identified in the stability manifold up to the action of $\widetilde{\text{GL}}^+_2(\R)$, answering a question asked by Macr\`{\i} and Stellari (see~\cite[Remark 4.9]{pertusi:some-remarks-fano-threefolds-index-two}). 
 
\begin{Thm}[Corollary~\ref{cor_barsigmavssigmaalphabeta}]
\label{thm_1}
Let $X$ be a cubic threefold. The stability condition $\overline{\sigma}$ on $\ku(X)$ defined in~\cite{macri:acm-bundles-cubic-3fold} is in the same orbit as the stability conditions $\sigma(\alpha,\beta)$ introduced in~\cite{bayer:stability-conditions-kuznetsov-component} with respect to the  $\widetilde{\emph{GL}}^+_2(\R)$-action.    
\end{Thm}

Our second result provides a detailed description of the moduli spaces of stable objects in $\ku(X)$ with minimal dimension with respect to any Serre-invariant stability condition. More precisely, let $M_X(v)$ and $M_X(w)$ be the moduli of slope-stable sheaves on $X$ with Chern characters
\begin{equation*}
    v = \left(1, 0, -\frac{1}{3}H^2, 0\right) \quad \text{and} \quad 
    w = \left(2, -H, -\frac{1}{6}H^2, \frac{1}{6}H^3\right). 
\end{equation*}
Note that $v$ is the numerical class of the ideal sheaf $\cI_{\ell}$ of a line $\ell$ in $X$ and in fact $M_X(v)$ is isomorphic to the Fano variety of lines on $X$. We also denote by $M_X(v-w)$ the moduli of large volume stable complexes on $X$ of class $v-w$; see Definition~\ref{def-large-stable}. 

\begin{Thm}[Theorem~\ref{thm-moduli-spaces}]\label{thm-introduction-2}
We have the isomorphisms 
\begin{equation*}
    \xymatrix@C+=3cm{
    M_X(v) \ar[r]_{\cong}^{L_{\cO_X}(- \otimes \cO(H))} & M_X(w) \ar[r]_{\cong}^{L_{\cO_X}(- \otimes \cO(H))} &  M_X(v-w).
    }
\end{equation*}
Moreover, the above three moduli spaces are isomorphic to the moduli of $\sigma$-stable objects of class $[\cI_{\ell}]$ in $\ku(X)$, where $\sigma$ is any Serre-invariant stability condition on $\ku(X)$. 
\end{Thm}

Finally, we apply Theorem~\ref{thm_1} to study the moduli space of Ulrich bundles of rank at least $2$ on $X$. Recall that an Ulrich bundle $E$ on $X$ is an arithmetically Cohen--Macaulay vector bundle such that the graded module $\oplus_{m \in \Z} H^0(X, E(mH))$ has $3 \rk(E)$ generators in degree $1$ (see Section~\ref{sec_Ulrich}). Ulrich bundles all lie in $\ku(X)$. 
Thus the moduli spaces of stable objects in $\ku(X)$ become a useful tool  to deduce properties of certain classical moduli spaces of semistable sheaves on $X$, such as non-emptyness and irreducibility. Applying these techniques, Lahoz, Macr\`{\i} and Stellari~\cite[Theorem B]{macri:acm-bundles-cubic-3fold} show that the moduli space $\mathfrak{M}^{sU}_{d}$ of stable Ulrich bundles of rank $d$ on $X$ is non-empty and smooth of the expected dimension $d^2+1$. They leave as an open question its irreducibility, which is our last result. 

\begin{Thm}[Theorem~\ref{thm_ulrich}]
\label{thm_2}
Let $X$ be a cubic threefold. The moduli space of Ulrich bundles of rank at least $2$ on $X$ is irreducible.
\end{Thm}

\subsection{Related works and motivation}
The interest in the study of Serre-invariant stability conditions on semiorthogonal components in the bounded derived category has grown recently, mostly due to  applications to the study of moduli spaces (see for instance~\cite{pertusi:some-remarks-fano-threefolds-index-two, LiuZhang}) and to the desire to better understand  the Kuznetsov component. For Fano threefolds of Picard rank $1$ and index $2$, the existence of Serre-invariant stability conditions on their Kuznetsov component is proved in~\cite{pertusi:some-remarks-fano-threefolds-index-two}, making use of the stability conditions constructed in~\cite{bayer:stability-conditions-kuznetsov-component}. In the upcoming paper~\cite{PR}, the same result is proved for the Kuznetsov component of a Gushel--Mukai threefold. On the other hand, in the recent paper~\cite{KuzPer}, the authors show the non-existence of Serre-invariant stability conditions on Kuznetsov components of almost all Fano complete intersections of codimension at least $2$. 

The assumptions in Theorem~\ref{thm_uniqueSinv_intro} require the category to be fractional Calabi--Yau with numerical Grothendieck group of rank $2$ generated by objects $E$ with small $\hom^1(E, E)$. The first condition allows one to control the phase of stable objects after the action of the Serre functor, and the other conditions make the category similar to the bounded derived category of a curve. Since in the case of curves of genus at least $2$ there is a unique orbit of stability conditions by Macr\`{\i}'s result~\cite{macri:stability-conditions-on-curves}, these are very natural conditions for one to expect the category to have a unique Serre-invariant stability condition.

Theorem~\ref{thm_uniqueSinv_intro} applies to cubic threefolds and to other Fano threefolds (see Remark~\ref{rem_applicationto Fano3}). Note that in these cases the uniqueness result has been recently proved independently by~\cite{zhang-hochschild:}. We explain in Section~\ref{subsec_cubicfourfolds} a related application in the case of very general cubic fourfolds.

In Proposition~\ref{prop.2-dim} we apply the method in~\cite{bayer:stability-conditions-kuznetsov-component} to construct stability conditions on $\ku(X)$ via the embedding of $\ku(X)$ in $\Db(\PP^2, \cB_0)$; see~\cite{macri:acm-bundles-cubic-3fold}. These stability conditions are Serre invariant, as shown in Section~\ref{sec-Srreinvariance}. 
Since the stability conditions constructed in~\cite{bayer:stability-conditions-kuznetsov-component} are also Serre invariant by~\cite[Corollary 5.5]{pertusi:some-remarks-fano-threefolds-index-two}, we deduce that all the stability conditions constructed on $\ku(X)$ up to now are Serre invariant. Theorem~\ref{thm_1} then follows  from Theorem~\ref{thm_uniqueSinv_intro}. An interesting question would be to understand whether the property of Serre invariance characterises the stability conditions on $\ku(X)$. This fact together with Theorem~\ref{thm_uniqueSinv_intro} would allow one to show that there is a unique orbit of stability conditions on $\ku(X)$, in analogy to the case of curves of genus at least $2$; see~\cite{macri:stability-conditions-on-curves}.  

Another open question is to generalise Theorem~\ref{thm-introduction-2} and Corollary~\ref{cor_moduliprojective} to further study moduli spaces of semistable objects in $\ku(X)$ with respect to a Serre-invariant stability condition like their projectivity or irreducibility. Some cases of small dimension have been studied in~\cite{pertusi:some-remarks-fano-threefolds-index-two, rota, feyz:desing, qin, LiuZhang}. 

The study of Ulrich bundles is a central theme in classical algebraic geometry and commutative algebra (see for instance~\cite{eisenbud:resultants} and~\cite{beauville:introUlrich} for a survey). The existence of Ulrich bundles have been shown on cubic threefolds by~\cite[Theorem B]{macri:acm-bundles-cubic-3fold}.
In the rank $2$ case, namely the case of instanton sheaves of minimal charge, we know a full description of the moduli space with class $2[\cI_\ell]$ for cubic threefolds, see \cite{macri:acm-bundles-cubic-3fold}, and more generally for Fano threefolds of Picard rank $1$, index $2$ and degree $d \geq 3$; see~\cite{qin, LiuZhang}. For the case of cubic fourfolds, we refer to~\cite{FaenziKim}. 

Theorem~\ref{thm_2} follows from an embedding of the moduli space of Ulrich bundles of rank $d \geq 2$ on $X$ in a moduli space of semistable objects in $\ku(X)$ with class $d[\cI_\ell]$, where $\cI_\ell$ is the ideal sheaf of a line in $X$. By Theorem~\ref{thm_1}, the latter moduli space can be described via an irreducible moduli space parametrising  Gieseker-semistable sheaves which are $\cB_0$-modules; it has the same dimension as 
$\mathfrak{M}^{sU}_{d}$.

\subsection{Plan of the paper}
In Section~\ref{sec_Bridgstabandtilting} we recall the definitions and basic properties of (weak) stability conditions, tilt stability, wall-crossing and a method to construct stability conditions on Kuznetsov components. Section~\ref{section-serre-invariant} is devoted to the proof of Theorem~\ref{thm_uniqueSinv_intro}. In Section~\ref{sec_kuzcubicthreefold} we review the construction of the stability conditions $\sigma(\alpha, \beta)$ from~\cite{bayer:stability-conditions-kuznetsov-component}, and then we prove Theorem~\ref{thm-introduction-2}. In Section~\ref{sec_kuzconicfibr} we extend the construction of stability conditions on $\ku(X)$ from~\cite{macri:categorical-invarinat-cubic-threefolds,macri:acm-bundles-cubic-3fold}, and we show they are Serre invariant. Finally, we deduce Theorem~\ref{thm_1}. In Section~\ref{sec_Ulrich} we prove Theorem~\ref{thm_2}.

\subsection*{Acknowledgments} 
We are very grateful to Chunyi Li, Emanuele Macr\`{\i}, Paolo Stellari, Richard Thomas, Xiaolei Zhao and Shizhuo Zhang for many useful conversations, and to Arend Bayer for suggesting the strategy used in Theorem~\ref{thm_uniqueSinv_intro}. We also thank an anonymous referee for useful comments.

\section{(Weak) Bridgeland stability conditions and tilting} \label{sec_Bridgstabandtilting}

In this section we recall the definitions and properties of (weak) Bridgeland stability conditions and wall-crossing. Our main reference is~\cite{bayer:the-space-of-stability-conditions-on-abelian-threefolds}. We also mention the stronger Bogomolov inequalities proved in~\cite{li:fano-picard-number-one} in the case of cubic threefolds. Finally, we explain
the method to construct stability conditions on the orthogonal complement of an exceptional collection introduced in~\cite{bayer:stability-conditions-kuznetsov-component}.  

\subsection{Review on (weak) stability conditions} \label{sec-review}
Let $\cT$ be a $\C$-linear triangulated category, and denote by $K(\cT)$ its Grothendieck group. A (weak) stability condition on $\cT$ is defined by giving the heart of a bounded t-structure and a (weak) stability function on it. Fix a finite-rank lattice $\Lambda$ with a surjective morphism $v \colon K(\cT) \twoheadrightarrow \Lambda$. We denote by $\Re[-]$ (resp.\ $\Im [-]$) the real (resp.\ imaginary) part of a complex number. 

\begin{Def}\label{def:weakstabfunction}
Let $\cA$ be the heart of a bounded t-structure on $\cT$. A group homomorphism $Z \colon \Lambda \rightarrow \C$ is a
\emph{$($weak$)$ stability function} on $\cA$ if for any $0 \neq E \in \cA$, we have $\Im Z(v(E)) \geq 0$, and in the case that $\Im Z(v(E)) = 0$, we have $\Re Z(v(E))< (\leq)\, 0$.
\end{Def}

By abuse of notation, we will write $Z(E)$ for an object $E \in \cA$ instead of $Z(v(E))$. The slope of $E \in \cA$ with respect to $Z$ is defined by
$$\mu_{Z}(E)= 
\begin{cases}
-\frac{\Re Z(E)}{\Im Z(E)} & \text{if } \Im Z(E) > 0, \\
+ \infty & \text{otherwise.}
\end{cases}$$
An object $F \in \cT$ is called \emph{$($semi$\,)$stable} with respect to the pair $(\cA , Z)$ if some shift $F[k]$ lies in the heart $\cA$ and for every proper subobject $F' \subset F[k]$ in $\cA$ we have 
$$\mu_{Z}(F') < (\leq) \ \mu_{Z}(F[k]/F').$$

\begin{Def}\label{def:wstab}
A \emph{$($weak$)$ stability condition} (with respect to $\Lambda$) on $\cT$ is a pair $\sigma=(\cA,Z)$, where $\cA$ is the heart of a bounded t-structure on $\cT$ and $Z$ is a (weak) stability function satisfying the following properties:
\begin{enumerate}
    \item \emph{HN property}: Every object of $\cA$ has a Harder--Narasimhan (HN) filtration with $\sigma$-semistable factors.
    \item  \emph{Support property}: There exists a quadratic form $Q$ on $\Lambda \otimes \R$ such that the restriction of $Q$ to $\ker Z$ is negative definite and $Q(E) \geq 0$ for all $\sigma$-semistable objects $E$ in $\cA$.
\end{enumerate}
\end{Def}

We denote by $\Stab_{\Lambda}(\cT)$ the space of stability conditions on $\cT$ with respect to $\Lambda$. This space is actually a complex manifold; see \cite{bridgeland:stability-condition-on-triangulated-category}.

If $\Lambda$ is the numerical Grothendieck group of $\cT$, we call $\sigma$ a \emph{numerical} (weak) stability condition. Any (weak) stability condition defines a slicing on $\cT$. 

\begin{Def}
\label{def_slicing}
The \textit{phase} of a $\sigma$-semistable object $E \in \cA$ is
$$\phi(E) \coloneqq \frac{1}{\pi}\text{arg}(Z(E)) \in (0,1].$$
If $Z(E)=0$, then $\phi(E)=1$. For $F=E[n]$, we set
$$\phi(E[n]):=\phi(E)+n.$$ 
The \textit{slicing} $\cP_{\sigma}$ of $\cT$ corresponding to $\sigma$ is a collection of full additive subcategories $\cP_{\sigma}(\phi) \subset \cT$ for $\phi \in \R$, such that:
\begin{enumerate}
\item[(i)] for $\phi \in (0,1]$, the subcategory $\cP_{\sigma}(\phi)$ is given by the zero object and all $\sigma$-semistable objects with phase $\phi$;
\item[(ii)] for $\phi+n$ with $\phi \in (0,1]$ and $n \in \Z$, we set $\cP_{\sigma}(\phi+n)\coloneqq \cP_{\sigma}(\phi)[n]$.
\end{enumerate} 
The HN property of (weak) stability condition $\sigma$ implies that for $0 \neq E \in \cT$, there exists a unique finite sequence of real numbers 
\begin{equation*}
    \phi^+(E) \coloneqq \phi_1 > \phi_2 > \cdots > \phi_n \eqqcolon \phi^-(E)
\end{equation*}
and a unique sequence of objects in $\cT$ 
\begin{equation*}
    0= E_0 \xrightarrow{f_1} E_1 \xrightarrow{f_2} E_2 \xrightarrow{f_3} \cdots \xrightarrow{f_{n-1}} E_{n-1} \xrightarrow{f_n} E_n= E
\end{equation*}
with cone$(f_i) \in \cP_\sigma(\phi_i)$. 
\end{Def}

\noindent It is clear from the definition that $\cA=\cP_\sigma(0,1]$, where the latter is the full subcategory of $\cT$ consisting of the zero object together with those objects $0 \neq E \in \cT$ which satisfy $0 < \phi^-(E) \leq \phi^+(E) \leq 1$.  

\pagebreak

On the stability manifold $\Stab_{\Lambda}(\cT)$, we have: 
\begin{enumerate}
\item  \emph{a right action of the universal covering space $\widetilde{\mathrm{GL}}^+_2(\R)$ of $\mathrm{GL}^+_2(\R)$}: for a stability condition $\sigma=(\cP,Z) \in \Stab_{\Lambda}(\cT)$ and $\tilde{g}=(g,M) \in \widetilde{\mathrm{GL}}^+_2(\R)$, where $g: \R \to \R$ is an increasing function such that $g(\phi+1)=g(\phi)+1$ and $M \in \mathrm{GL}^+_2(\R)$, we define $\sigma \cdot \tilde{g}$  to be the stability condition $\sigma'=(\cP',Z')$ with $Z'=M^{-1} \circ Z$ and $\cP'(\phi)=\cP(g(\phi))$ (see~\cite[Lemma 8.2]{bridgeland:stability-condition-on-triangulated-category});
\item \emph{a left action of the group of linear exact autoequivalences $\Aut_{\Lambda}(\cT)$ of $\cT$, whose induced action on $K(\cT)$ is compatible with $v$}: for $\Phi \in \Aut_{\Lambda}(\cT)$ and $\sigma \in \Stab_{\Lambda}(\cT)$, we define $\Phi \cdot \sigma=(\Phi(\cP), Z \circ \Phi_*^{-1})$, where $\Phi_*$ is the automorphism of $K(\cT)$ induced by $\Phi$.
\end{enumerate}

\subsection{Tilt stability} \label{sec-tiltstab}
An important tool in the construction of weak stability conditions is the procedure of tilting a heart. Let $\sigma=(\cA,Z)$ be a weak stability condition on $\cT$ and $\mu\in\R$. We consider
the following subcategories of $\cA$:
\begin{align*}
    \cT_\sigma^{\mu}\coloneqq \{E \in \cA|\text{ every Harder--Narasimhan factor $F$ of $E$ has $\mu_Z(F)>\mu$}\},\\
     \cF_\sigma^{\mu}\coloneqq \{E \in \cA|\text{ every Harder--Narasimhan factor $F$ of $E$ has $\mu_Z(F)\leq \mu$}\}.
\end{align*}

\begin{Thm}[\cite{happel:tilting-in-abelian-categories-and-quasitilted-algebras}]\label{pd:tiltheart}
The category
$$\cA_\sigma^\mu\coloneqq\langle \cT_\sigma^{\mu},\cF_\sigma^{ \mu}[1]\rangle_{\text{\rm extension closure}}$$
is the heart of a bounded t-structure on $\cT$.
\end{Thm} 

\noindent We say that the heart $\cA^\mu_{\sigma}$ is obtained by tilting $\cA$ with respect to the weak stability condition $\sigma$ at slope $\mu$.

The above construction applies to the case of the bounded derived category $\cT=\Db(X)$ of coherent sheaves on a smooth projective variety $X$, to define tilt stability conditions. Fix an ample divisor $H$ on $X$, and set $n:=\text{dim}(X)$. The pair $\sigma_H=(\Coh(X),\ Z_H)$, where $Z_H= -\ch_1H^{n-1}+\ch_0H^n$, is a weak stability condition on $\Db(X)$, known as \emph{slope stability}, with respect to the rank $2$ lattice generated by the elements of the form $(H^n\ch_0(E), H^{n-1} \ch_1(E))$ for $E \in \Db(X)$; see \cite[Example 2.8]{bayer:stability-conditions-kuznetsov-component}. In particular, the slope of a coherent sheaf $E$ on $X$ is defined by
$$\mu_H(E)=
\begin{cases}
\frac{H^{n-1}\ch_1(E)}{H^n\ch_0(E)} & \text{if } \ch_0(E) > 0,\\
+\infty & \text{otherwise}.
\end{cases}$$
Any $\mu_{H}$-semistable sheaf $E$ satisfies the Bogomolov--Gieseker inequality
\begin{equation*}
\Delta_{H}(E):=(H^{n-1}\ch_{1}(E))^{2}-2H^{n}\ch_{0}(E)\cdot H^{n-2}\ch_{2}(E)\geq 0.
\end{equation*}
We observe that if $n=1$, then $\sigma_H$ is a stability condition.

Now given $\beta\in \R$, we denote by $\Coh^{\beta}(X)$ the heart of a bounded $t$-structure obtained by tilting the weak stability condition $\sigma_H$ at the slope $\mu_H=\beta$ (see Theorem~\ref{pd:tiltheart}). 
For $E \in \Db(X)$, the twisted Chern character is defined by $\ch^{\beta}(E):=e^{-\beta H}\ch(E)$. Explicitly, the first three terms are
$$\ch_0^{\beta}(E):=\ch_0(E), \quad  \ch_1^{\beta}(E):=\ch_1(E)-\beta H \ch_0(E)$$
and
$$\ch_2^{\beta}(E):= \ch_2(E) -\beta H \ch_1(E) +\frac{\beta^2 H^2}{2}\ch_0(E).$$

\begin{Prop}[\protect{\cite[Proposition 2.12]{bayer:stability-conditions-kuznetsov-component}}]
\label{first-tilting-wsc}
There is a continuous family of weak stability conditions parametrised by $\R_{>0}\times \R$, given by
$$(\alpha, \beta) \in \R_{>0}\times \R \mapsto \sigma_{\alpha, \beta}=(\emph{Coh}^{\beta}(X), Z_{\alpha, \beta})$$
with
\begin{align*}
    Z_{\alpha, \beta}(E) \colon \Lambda \cong \Z^3\ &\ \rightarrow \ \mathbb{C}\\
    \left(H^3\ch_0(E), H^2\ch_1(E), H\ch_2(E)\right) \ &\ \mapsto\  \frac{1}{2}\alpha^2 H^{n}\ch_{0}^{\beta}(E)-H^{n-2}\ch_{2}^{\beta}(E)
+i H^{n-1}\ch_{1}^{\beta}(E). 
\end{align*}
\end{Prop}

The weak stability condition $\sigma_{\alpha, \beta}$ is called \emph{tilt stability}, and if $n=2$, then it defines a stability condition on $\Db(X)$ (see~\cite{bridgeland:K3-surfaces, AB}). There is a region in the upper half-plane where the notions of $\sigma_{\alpha, \beta}$-stability and $\mu_H$-stability are closely related. 

\begin{Lem}[\protect{\cite[Lemma 2.7]{bayer:the-space-of-stability-conditions-on-abelian-threefolds}}]
\label{lem-large-volume}
Assume $E \in \Coh(X)$ is a $\mu_H$-stable sheaf of positive rank. Then $E$ is an element of $\,\Coh^{\beta}(X)$ for $\beta < \mu_H(E)$ and is $\sigma_{\alpha, \beta}$-stable for $\alpha \gg 0$. 
\end{Lem}

We now recall the notions of walls and chambers for tilt stability.

\begin{Def}
Fix $v \in \Lambda$. A \textit{numerical wall} for $v$ is a subset of the upper half-plane of the form
$$W(v, w)= \lbrace (\alpha,\beta) \in \R_{>0} \times \R \colon \mu_{\alpha,\beta}(v)=\mu_{\alpha,\beta}(w)  \rbrace$$
with respect to a vector $w \in \Lambda$. We say a point $(\alpha, \beta) \in W(v, w)$ is on an \textit{actual wall} for $v$ if and only if there is an object $E \in \Coh^{\beta}(X)$ of class $v(E) = v$ which is strictly $\sigma_{\alpha, \beta}$-semistable and unstable on one side of the wall.   

A \textit{chamber} for class $v$ is a connected component in the complement of the union of walls in the upper half-plane. 
\end{Def}
Tilt stability satisfies well-behaved wall and chamber structure, in the sense that walls with respect to a class $v \in \Lambda$ are locally finite. Properties of walls are described by the following theorem, which was first proved in~\cite{maciocia:the-walls-projective-spaces} (see also~\cite{Sch}). 

\begin{Thm}[Bertram's nested wall theorem]
Fix $v \in \Lambda$.
\begin{enumerate}
\item Numerical walls for $v$ are either nested semicircles centred on the $\beta$-axis, or a vertical ray
parallel to the $\alpha$-axis.
\item If two numerical walls intersect, then they are the same. If a numerical wall contains a point defining an actual wall, then the numerical wall is an actual wall.
\item If $\ch_0(v) \neq 0$, then there is a unique vertical numerical wall given by $\beta= \mu_H(v)$. If $\ch_0(v)=0$ and $H^{n-1}\ch_1(v) \neq 0$, then there is no vertical wall.
\end{enumerate}
\end{Thm}

\subsection{Cubic threefolds}  
We know any slope-semistable sheaf and, more generally, any $\sigma_{\alpha, \beta}$-semistable object satisfies $\Delta_H(E) \geq 0$; see \cite[Theorem 3.5]{bayer:the-space-of-stability-conditions-on-abelian-threefolds}. In the case of a smooth complex cubic threefold $X$, we have stronger Bogomolov inequalities as follows. 

\begin{Prop}[\protect{\cite[Corollary 3.4]{li:fano-picard-number-one}}]\label{prop-li-ch2}
Let $E$ be a slope-semistable coherent sheaf of positive rank on $X$. If $\abs{\mu_H(E)} \leq \frac{1}{2}$, then $\ch_2(E) \leq 0$, and if $\frac{1}{2} < \abs{\mu_H(E)} \leq 1$, then
\begin{equation*}
    H\ch_2(E) \leq \abs{H^2\ch_1(E)} - \frac{1}{2}H^3\ch_0(E). 
\end{equation*}
\end{Prop}

\begin{Thm}\label{thm-li-ch3}
For any $\sigma_{\alpha, \beta}$-semistable object $E \in \Db(X)$, we have 
\begin{equation*}
    0 \leq Q_{\alpha, \beta}(E) \coloneqq \left( \frac{\alpha^2}{2} + \frac{\beta^2}{2}
    \right)\left(C_1^2-2C_0C_2\right) +\beta \left(3C_0C_3-C_1C_2\right) +\left(2C_2^2-3C_1C_3\right),
\end{equation*}
where $C_i:=\ch_i(E).H^{3-i}$.
\end{Thm}

\begin{proof}
Conjecture 5.3 of \cite{bayer:the-space-of-stability-conditions-on-abelian-threefolds} is proved for cubic threefolds in~\cite[Theorem 0.1]{li:fano-picard-number-one}. Thus~\cite[Theorem 5.4]{bayer:the-space-of-stability-conditions-on-abelian-threefolds} implies that~\cite[Conjecture 4.1]{bayer:the-space-of-stability-conditions-on-abelian-threefolds} holds on cubic threefold. One can easily show that its statement is equivalent to $Q_{\alpha, \beta}(E) \geq 0$. 
\end{proof}

\subsection{Induced stability conditions}
We finish this section by recalling induced stability conditions on an admissible subcategory. Let $\cT$ be a $\C$-linear triangulated category with Serre functor $S$. Assume $\cT$ has a semiorthogonal decomposition of the form 
$$
\cT
=\langle \cD, E_1,\ldots ,E_m\rangle,
$$
where $\{E_1, \dots, E_m\}$ is an exceptional collection and $\cD:=\langle E_1, \dots, E_m \rangle^{\perp}$. The next proposition gives a criterion  to construct a stability condition on $\cD$ from a weak stability condition on $\cT$.

\begin{Prop}[\protect{\cite[Proposition 5.1]{bayer:stability-conditions-kuznetsov-component}}]
\label{prop_inducstab}
Let $\sigma=(\cA, Z)$ be a weak stability condition on $\cT$. 
Assume that the exceptional collection $\{E_1,\ldots ,E_m\}$ satisfies the following conditions: 
\begin{enumerate}
\item $E_i\in \cA$;
\item $S(E_i)\in \cA[1]$; 
\item  $Z(E_i)\neq0$ for all $i=0,1, \ldots, m$.
\end{enumerate}
If moreover there are no objects $0\neq F\in \cA':=\cA\cap \cD$ with $Z(F)=0$,
then the pair $\sigma'=(\cA',Z|_{K(\cA)})$ is a stability condition on $\cD$.
\end{Prop}

\section{Serre-invariant stability conditions}\label{section-serre-invariant}
Let $\cT$ be a triangulated category which is linear of finite type over a field $K$; \textit{i.e.} $\oplus_i\Hom(E, F[i])$ is a finite-dimensional vector space over $K$. Assume that $\cT$ has Serre functor~$S$. Since $S$ is a linear exact autoequivalence, we can have the following definition.   

\begin{Def} \label{def-Sinvstab}
A stability condition $\sigma$ on $\cT$ is \emph{Serre invariant} (or \emph{$S$-invariant}\,) if $S\cdot \sigma = \sigma \cdot \tilde{g}$ for some $\tilde{g} \in \widetilde{\text{GL}}_2^+(\R)$. 
\end{Def}

Assume $\cT$ satisfies the following conditions: 
\begin{enumerate}[label=\rm{(}\textbf{C\arabic*}\rm{)},ref=\textbf{C\arabic*}]
	\item\label{C1}The Serre functor $S$ of $\cT$ satisfies $S^r = [k]$ when $0 < k/r < 2$ or $r=2$ and $k=4$. 
	\item\label{C2} The numerical Grothendieck group $\mathcal{N}(\cT)$ is of rank $2$, and we have the inequality $\ell_{\cT} \coloneqq \max\{\chi(v, v) \colon 0 \neq v \in \mathcal{N}(\cT) \} < 0$.
	\item\label{C3} There is an object $Q \in \cT$ satisfying 
	 	\begin{equation}\label{minimal}
	   -\ell_{\cT} +1 \leq  \hom^1(Q, Q) <  -2\ell_{\cT} +2. 
	\end{equation}
	If $k= 2r =4$, there are two objects $Q_1, Q_2 \in \cT$ satisfying \eqref{minimal} such that $Q_1$ is not isomorphic to $Q_2$ or $Q_2[1]$, $\hom(Q_2, Q_1) \neq 0$ and $\hom(Q_1, Q_2[1]) \neq 0$. 
\end{enumerate}
We can slightly relax~\eqref{C3} in the case $k= 2r =4$; see Lemma~\ref{lem-relax}. 

The goal of this section is to prove the following theorem. 

\begin{Thm} \label{thm_uniqueSinv}
Let $\sigma_1$ and $\sigma_2$ be $S$-invariant numerical stability conditions on a triangulated category $\cT$ satisfying the above conditions~\eqref{C1},~\eqref{C2} and~\eqref{C3}. Then there exists a $\tilde{g} \in \widetilde{\emph{GL}}^+(2, \mathbb{R})$ such that $\sigma_1= \sigma_2 \cdot \tilde{g}$. 
\end{Thm}

\begin{Rem} \label{rmk_CY2}
It is easy to see that if the Serre functor satisfies $S=[2]$, then every stability condition on $\cT$ is Serre invariant.
\end{Rem}

Before proving Theorem~\ref{thm_uniqueSinv}, we apply similar arguments as in~\cite{pertusi:some-remarks-fano-threefolds-index-two} to investigate some of the properties of $S$-invariant stability conditions.

\begin{Prop}\label{prop-PY}
Let $\cT$ be a triangulated category satisfying conditions~\eqref{C1} and~\eqref{C2} above, and let $\sigma = (Z, \cA)$ be an $S$-invariant stability condition on $\cT$. 
\begin{enumerate}
    \item If $E$ is a $\sigma$-semistable object of phase $\phi(E)$, then $\phi(S(E)) \leq \phi(E) +2$. The inequality is strict if\, $k/r <2$. \label{a} 
    \item \label{Mukai}{\rm(}Weak Mukai lemma{\rm)} Let $A \to E \to B$ be an exact triangle in $\cT$ such that $\Hom(A,B)=0$ and the $\sigma$-semistable factors of $A$ have phases greater than $($or equal to in case $k/r <2)$ the phases of the $\sigma$-semistable factors of $B$. Then 
    \begin{equation}
        \label{eq_MukaiLemma}
        \hom^1(A,A)+ \hom^1(B,B) \leq \hom^1(E,E).
    \end{equation}  
    \item\label{min-hom1} Any non-zero object $E \in \cT$ satisfies
    \begin{equation*}
        \hom^1(E, E) \geq -\ell_{\cT} +1. 
    \end{equation*}
    \item\label{min-phase} A $\sigma$-semistable object $E$ satisfies $\phi(E) +1 \leq \phi(S(E))$. The inequality is strict if $E$ is $\sigma$-stable.
    \item \label{minimal-stable} If\, $E \in \cT$ satisfies $-\ell_{\cT} +1 \leq  \hom^1(E, E) <  -2\ell_{\cT} +2$, then it is $\sigma$-stable.
\end{enumerate}
\end{Prop}

\begin{proof}
Since $\sigma$ is an $S$-invariant stability condition on $\cT$, there is a $(g, M) \in \widetilde{\text{GL}}_2^+(\R)$, where $g\colon \R \rightarrow \R$ is an increasing map with $g(\phi+1) = g(\phi) +1$, such that $S\cdot\sigma = \sigma\cdot\tilde{g}$. Thus $\phi(S(F)) = g(\phi(F))$ and for any $n >0$,
\begin{equation}\label{equality}
    \phi(S^n(F)) = g(\phi(S^{n-1}(F))) = g^2(\phi(S^{n-2}(F))) = \cdots = g^n(\phi(F)). 
\end{equation}
	If $\phi(S(F)) = g(\phi(F))\ >(\geq)\  \phi(F) +2$, then since $g$ is increasing,
	\begin{equation*}
	   \phi(F) +k = \phi(S^r(F)) \overset{\eqref{equality}}{=} g^r(\phi(F)) >(\geq) g^{r-1}(\phi(F) +2) = g^{r-1}(\phi(F)) +2  >(\geq)
	    \cdots >(\geq) \phi(F) +2r, 
	\end{equation*}
	which gives $k/r \ >(\geq)\ 2$, leading to a contradiction. This completes the proof of part \eqref{a}. 
\bigskip

To prove part \eqref{Mukai}, let $A \rightarrow E \rightarrow B$ be an exact triangle in $\cT$ with $\Hom(A, B) = 0$. If $\phi^+(B) < \phi^{-}(A)$, then part \eqref{a} implies $\phi^+(S(B)) -2 \leq \phi^+(B)$, which gives
\begin{equation}\label{phase}
\phi^+(S(B)) < \phi^-(A[2]).    
\end{equation}
If $k/r < 2$, we may assume $\phi^+(B) \leq \phi^{-}(A)$ because by part \eqref{a} again, \eqref{phase} holds. Then
\begin{equation}\label{van}
\hom(B, A[2]) = \hom(A[2], S(B)) = 0,
\end{equation}
and the same argument as in~\cite[Lemma 2.4]{bayer:derived-auto} implies the claim \eqref{eq_MukaiLemma}.

 \bigskip

To prove part \eqref{min-hom1}, first assume $E$ is $\sigma$-semistable. Up to shift, we may assume $E \in \mathcal{P}(0, 1]$. Thus $\hom(E, E[i]) = 0$ for $i <0$. Also by part \eqref{a}, $\phi(S(E))\ \leq \ 3$; thus we have $\hom(E, E[i]) = \hom(E[i], S(E)) = 0$ for $i \geq 3$. Hence 
\begin{equation}\label{hom}
    \hom(E, E[1]) = -\chi(E, E) + \hom(E, E) + \hom(E, E[2]) \geq -\ell_{\cT} +1, 
\end{equation}
as claimed. If $E$ is not $\sigma$-semistable, then let $E_1 \rightarrow E \rightarrow E/E_1$ be the first piece in its HN filtration. Then applying weak Mukai lemma in part \eqref{Mukai} to this sequence gives
\begin{equation*}
    \hom^1(E_1, E_1) \leq \hom^1(E, E).
\end{equation*}
Thus the claim in part \eqref{min-hom1} follows by the first part of the argument because $E_1$ is $\sigma$-semistable. 
 
\bigskip 

Take a $\sigma$-semistable object $E$. Part \eqref{min-hom1} implies that $\hom(E[1] , S(E))=\hom(E, E[1]) \neq 0$; thus $\phi(E) +1 \leq \phi(S(E))$ as claimed in part \eqref{min-phase}. 
If $E$ is $\sigma$-stable, then $S(E)$ is also $\sigma$-stable. Since there is a non-zero map from $E$ to $S(E)[-1]$, they cannot have the same phase. Note that $S(E) \neq E[1]$; otherwise $\hom^1(E, E) = \hom(E, E) =1$, which is not possible by part \eqref{min-hom1}.
 
\bigskip

Finally, to prove part \eqref{minimal-stable}, take an object $E \in \cT$ satisfying the minimality condition
$$-\ell_T +1 \leq \hom^1(E, E) < -2\ell_T +2$$
on $\hom^1(E, E)$. Assume for a contradiction that $E$ is not $\sigma$-semistable, and let $A \rightarrow E \rightarrow B$ be the first piece in the HN filtration of $E$. So $A$ is $\sigma$-semistable and
\begin{equation*}
    \phi(A) = \phi^+(E) > \phi^+(B). 
\end{equation*}
Thus $\hom(A, B) = 0$, and we can apply weak Mukai lemma in part \eqref{Mukai}. Then part \eqref{min-hom1} implies 
\begin{equation*}
2(-\ell_{\cT}+1) \leq \hom^1(A, A)+\hom^1(B, B) \leq \hom^1(E, E),   
\end{equation*}
which is not possible by our assumption.

Now assume $E$ is strictly $\sigma$-semistable. First suppose $k/r <2$. Let $Q \hookrightarrow E$ be a $\sigma$-stable factor of $E$ with quotient $E_1 \in \mathcal{P}(\phi(E))$. If $\hom(Q, E_1) \neq 0$, then the $\sigma$-stability of $Q$ implies that there is an embedding $Q \hookrightarrow E_1$ with quotient $E_2 \in \mathcal{P}(\phi(E))$. By continuing this process, we get an exact triangle $A \rightarrow E \rightarrow B$ in $\mathcal{P}(\phi(E))$ such that $A$ is $S$-equivalent to $\oplus_k Q$ for some $k >0$ (\textit{i.e.}\ all $\sigma$-stable factors of $A$ are isomorphic to $Q$) and $\hom(A, B) = 0$. If $B \neq 0$, then the weak Mukai lemma and part \eqref{min-hom1} lead to a contradiction as above. 
Thus $[E] = k [Q] \in \mathcal{N}(\cT)$ and $\chi(E, E) = k^2\chi(Q, Q) \leq k^2\ell_{\cT}$. This gives 
\begin{equation*}
    \hom^1(E, E) = -\chi(A, A) + \hom(A, A) + \hom(A, A[2]) \geq -k^2\ell_{\cT} +1. 
\end{equation*}
Then our assumption implies $-k^2\ell_{\cT} +1 < -2\ell_{\cT} +2$, which is possible only if $k =1$ because $\ell_T \leq -1$. Thus $E$ is $\sigma$-stable. This completes the proof of \eqref{minimal-stable} for the case $k/r <2$.

If $E$ is strictly $\sigma$-semistable and $k= 2r=4$, we apply  Lemma~\ref{lem-posibilities} below. In  Lemma~\ref{lem-posibilities}\eqref{ii}, the same argument as in the weak Mukai lemma in part \eqref{Mukai} implies that 
\begin{equation*}
    \hom^1(E_1, E_1) + \hom^1(E_2, E_2) \leq \hom^1(E, E),
\end{equation*}
which is not possible by the minimality assumption on $\hom^1(E, E)$. Thus  Lemma~\ref{lem-posibilities}\eqref{i}  happens; \textit{i.e.}\ $E$ is $S$-equivalent to $G^{\oplus {k_1}} \oplus S(G)^{\oplus {k_2}}[-2]$ for a stable object $G$. Hence $\chi(E, E) = (k_1+k_2)^2\chi(G, G)$ because 
\begin{equation*}
    \chi(G, G) = \chi(G, S(G)) = \chi(S(G), S(G)) = \chi(S(G) , S^2(G)) = \chi(S(G) , G). 
\end{equation*}
Therefore, \eqref{hom} gives
\begin{equation*}
\hom^1(E, E) \geq -\chi(E, E) +1 = -(k_1+k_2)^2\chi(G, G) +1. 
\end{equation*}  
If $k_1+k_2 \geq 2$, then 
\begin{equation*}
    -2\ell_T +1 > \hom^1(E, E) \geq -4\ell_T +1,
\end{equation*}
which implies $0 > -2\ell_T$, leading to a contradiction. Thus $k_1+k_2 =1$; \textit{i.e.}\ $E$ is isomorphic to either $G$ or $S(G)[-2]$, so $E$ is $\sigma$-stable as claimed.  
\end{proof}
 
\begin{Lem}\label{lem-posibilities}
	Suppose $k=2r =4$ and $E$ is a strictly $\sigma$-semistable object. Then either 
	\begin{enumerate}
	    \item\label{i} $E$ is $S$-equivalent to $G^{\oplus {k_1}} \oplus S(G)^{\oplus {k_2}}[-2]$ for a stable object $G$ which has same phase as $S(G)[-2]$ if\, $k_2 \neq 0$, or
	    \item\label{ii} there is a non-trivial triangle $E_1 \rightarrow E \rightarrow E_2$ of $\sigma$-semistable objects of the same phase such that $\hom(E_1, E_2) = 0$ and $\hom(E_2, E_1[2]) =0$. 
	\end{enumerate}
\end{Lem}

\begin{proof}
Considering the Jordan--H\"older filtration, we can write the short exact sequence $A \rightarrow E \rightarrow G$, where $G$ is stable of the same phase as $E$. If $\hom(A, G) \neq 0$, then we gain a surjection $A \twoheadrightarrow G$ with kernel $A_1$: 
\begin{equation*}
    \xymatrix{
    & & G\ar[d]\\
    A_1 \ar[r]\ar[d]&E\ar[r]\ar[d] &F_1\ar[d]\\
    A \ar[r]\ar[d]&E \ar[r] &G\\
    G\rlap{.}&&
    }
\end{equation*}
Thus the cone of $A_1 \to E$, denoted by $F_1$, is $S$-equivalent to $G^{\oplus 2}$. If $\hom(A_1, G) \neq 0$, we continue the above argument to reach a triangle $A_{m} \rightarrow E \rightarrow F_{m}$ such that $\hom(A_{m} , G) = 0$ and $F_{m}$ is $S$-equivalent to $G^{\oplus\, m+1}$. If $A_{m} = 0$, then $E$ is $S$-equivalent to $G^{\oplus\, m+1}$ as in case~\eqref{i}.   

We know $S(G)$ is $\sigma$-stable and $\phi(S(G)) \leq \phi(G) +2$ by Proposition~\ref{prop-PY}, part \eqref{a}. If $\phi(S(G)) < \phi(G) +2$, then 
\begin{equation*}
    \hom(G , A_{m}[2]) = \hom(A_{m}[2], S(G)) =0.   
\end{equation*}
Hence if $\phi(S(G)) < \phi(G) +2$ or $\hom(A_{m}, S(G)[-2]) = 0$ , then $\hom(F_{m} , A_{m}[2]) = 0$; thus we get a triangle as in case~\eqref{ii}. But if $\phi(S(G)[-2]) = \phi(G)$ and $\hom(A_{m}, S(G)[-2]) \neq 0$, then there is a commutative diagram
\begin{equation*}
    \xymatrix{
    & & S(G)[-2]\ar[d]\\
    B_1 \ar[r]\ar[d]& E\ar[r]\ar[d] & E_1\ar[d]\\
    A_{m}\ar[r]\ar[d]&E \ar[r]&F_{m}\\
    S(G)[-2]\rlap{,}&&
    }
\end{equation*}
where $E_1$ is $S$-equivalent to  $G^{\oplus\,m+1} \oplus S(G)[-2]$.  
If $\hom(B_1, S(G)[-2]) \neq 0$, we continue the above argument to get a triangle $B_{n} \rightarrow E \rightarrow E_{n}$ such that $E_{n}$ is $S$-equivalent to the sum $G^{\oplus\, m+1} \oplus S(G)^{\oplus\, n }[-2]$ and $\hom(B_{n}, S(G)[-2]) = 0$.
If $B_{n} = 0$, then we are in case~\eqref{i} and if $B_{n} \neq 0$, we continue the above process to obtain an exact triangle $E_1 \rightarrow E \rightarrow E_2$ such that $E_2$ is $S$-equivalent to  $G^{\oplus\, k_1} \oplus S(G)^{\oplus\, k_2 }[-2]$ and 
\begin{equation*}
    \hom(E_1, G) = 0 = 
    \hom(E_1, S(G)[-2]). 
\end{equation*}
This implies that 
\begin{equation*}
    \hom(E_1, E_2) = 0. 
\end{equation*}
Moreover, we have $ \hom(E_2, E_1[2]) = \hom(E_1, S(E_2)[-2])$, and $S(E_2)[-2]$ is $S$-equivalent to $S(G)[-2]^{\oplus k_1} \oplus S^2(G)^{\oplus k_2}[-4] = S(G)[-2]^{\oplus k_1} \oplus G^{\oplus k_2}$; thus 
$\hom(E_1, S(E_{2})[-2]) = 0$ as claimed in case~\eqref{ii}. 
\end{proof}

Now assume our triangulated category $\cT$ satisfies  condition~\eqref{C3}. If $k/r <2$, fix an object $Q$ satisfying \eqref{minimal}. By Proposition~\ref{prop-PY}\eqref{minimal-stable}, both $Q$ and $S(Q)[-2]$ are stable. Moreover, parts \eqref{a} and \eqref{min-phase} of the proposition imply 
\begin{equation*}
\phi(Q) -1 < \phi(S(Q)[-2]) < \phi(Q). 
\end{equation*}
Hence, up to $\widetilde{\text{GL}}_2^+(\R)$-action, we may assume $Q$ is stable of phase $1$ and $S(Q)[-2]$ is of phase~$\frac{1}{2}$. 

If $k= 2r =4$, we fix two objects $Q_1$ and $Q_2$ described in condition~\eqref{C3}. By Proposition~\ref{prop-PY}\eqref{minimal-stable}, they are $\sigma$-stable. Since $\hom(Q_2, Q_1) \neq 0$, we obtain $\phi(Q_2) < \phi(Q_1)$, and since $\hom(Q_1, Q_2[1]) \neq 0$, we get $\phi(Q_1) -1 < \phi(Q_2)$. Thus up to $\widetilde{\text{GL}}_2^+(\R)$-action, we may assume $Q_1$ is of phase $1$ and $Q_2$ is of phase $1/2$. From now on, when $k/r <2$, we define 
\begin{equation*}
    Q_1  \coloneqq Q \quad \text{and} \quad Q_2 \coloneqq S(Q)[-2].
\end{equation*}
As a consequence, we have the following. 

\begin{Lem}\label{lem-central charge}
Let $\sigma = (Z, \cA)$ be an $S$-invariant stability condition on $\cT$. Up to $\widetilde{\emph{GL}}_2^+(\R)$-action, we may assume $Z(Q_1) =-1$ and $Z(Q_2) = i$. 
\end{Lem}

\begin{proof}[Proof of Theorem~\ref{thm_uniqueSinv}] 
  Lemma~\ref{lem-central charge} implies that after action by
  $\widetilde{\text{GL}}^+(2, \mathbb{R})$, we can assume $\sigma_1$ and $\sigma_2$ have the same central charge $Z$ given by the bijective group homomorphism
	\begin{equation*}
	Z \colon \mathcal{N}(\cT) \rightarrow \mathbb{C}, \quad Z([Q_1]) = -1, \; Z([Q_2]) = i. 
	\end{equation*}
	By Proposition~\ref{prop-PY}\eqref{min-hom1}, any $\sigma_i$-semistable object $E$ satisfies $\hom^1(E, E) \geq -\ell_{\cT} +1$. We show by induction on $\hom^1(E, E)$ that if $E$ is a $\sigma_1$-(semi)stable object in $P_{\sigma_1}(0, 1]$ with phase $\phi_{\sigma_1}(E)$, then 
\begin{enumerate}[label=(\roman*)]
    \item\label{first} $E$ is $\sigma_2$-(semi)stable, and 
    \item\label{second} $E$ has the same phase as with respect to $\sigma_1$; \textit{i.e.}\ $\phi_{\sigma_2}(E) = \phi_{\sigma_1}(E)$.  
\end{enumerate}

\textbf{Step 1.} (Base of the induction) If $\hom^1(E, E)$ is minimal, \textit{i.e.}\ 
\begin{equation*}
    -\ell_T +1 \leq \hom^1(E, E) < -2\ell_T +2, 
\end{equation*}
Proposition~\ref{prop-PY}\eqref{minimal-stable} implies that $E$ is $\sigma_i$-stable. Thus it remains to show that if $E$ is in the heart $P_{\sigma_1}(0, 1]$ with phase  $\phi_{\sigma_1}(E)$, then $E \in P_{\sigma_2}(0, 1]$ with $\phi_{\sigma_2}(E) = \phi_{\sigma_1}(E)$.  

To do this, note that we only need to show $E \in P_{\sigma_2}(0, 1]$. Namely, in that case, since $\sigma_1$ and $\sigma_2$ have the same central charge, $E$ must have the same phase with respect to them. Assume otherwise, so there is  a non-zero $m \in \Z$ such that $E[2m] \in P_{\sigma_2}(0, 1]$. In fact, since $\sigma_1$ and $\sigma_2$ have the same central charge, an even shift of $E$ lies in the heart for $\sigma_2$. By Lemma~\ref{lem-central charge}, $E$ is not isomorphic to $Q_1$ or $Q_2$ since $Q_1$ and $Q_2$ are in the heart of $\sigma_1$ and $\sigma_2$. We want to show that 
\begin{equation}\label{claim1}
\chi(Q_1, E) =\chi(Q_2, E) = 0.    
\end{equation}
Indeed, since $E[i] \in P_{\sigma_1}(i, i+1]$ and $E[2m] \in P_{\sigma_2}(0, 1]$, we obtain   
\begin{equation*}
\hom(Q_1, E[i]) = \hom(Q_1, E[2m][i-2m]) = 0 \quad \text{when $i \leq \max\{0, 2m\}$}.
\end{equation*}
Proposition~\ref{prop-PY}\eqref{a},\eqref{min-phase} give $S(Q_1)[-2] \in P_{\sigma_i}(0, 1]$; thus
\begin{equation*}
\hom(Q_1, E[i]) = \hom(E[-2 +i], S(Q_1)[-2]) = \hom(E[2m][-2+i -2m] , S(Q)[-2]) = 0
\end{equation*}
if $i > \min \{2 , 2m+2\}$. Hence, in total we obtain $\chi(Q_1, E) = 0$. Similarly, we have 
\begin{equation*}
\hom(Q_2, E[i]) = \hom(Q_2, E[2m][i-2m]) = 0 \quad \text{when $i < \max\{0, 2m\}$}.
\end{equation*}
The above vanishing holds for $i = \max\{0, 2m\}$ if $\phi_{\sigma_1}(E) \in (0, 1/2]$. Moreover, Proposition~\ref{prop-PY}\eqref{a},~\eqref{min-phase} imply that $S(Q_2)[-2] \in P_{\sigma_i}(-1/2, 1/2]$. Thus
\begin{equation*}
    \hom(Q_2, E[i]) = \hom(E[i-2], S(Q_2)[-2]) = \hom(E[2m][i-2-2m], S(Q_2)[-2]) = 0
\end{equation*}
if $i > \min \{2 , 2m+2\}$. Also, the above vanishing holds for $i = \min \{2 , 2m+2\}$ when $\phi_{\sigma_1}(E) \in (1/2, 1]$. Therefore, $\chi(Q_2, E) = 0$, as claimed in \eqref{claim1}.  

Since $Z(Q_1)$ and $Z(Q_2)$ are linearly independent, the classes $[Q_1]$ and $[Q_2]$ are also linearly independent in $\mathcal{N}(\cT)$. Hence \eqref{claim1} implies that $\chi([F], [E]) = 0$ for any $[F] \in \mathcal{N}(\cT)$; thus $[E] = 0$ by the definition of $\mathcal{N}(\cT)$, leading to a contradiction. This completes the proof for the case that $\hom^1(E, E)$ is minimal.

\bigskip

\textbf{Step 2.} Now assume $\hom^1(E, E) \geq -2\ell_{\cT} +2$. There are three possibilities: 
\begin{enumerate}[label=(\roman*),ref=\roman*]
    \item\label{s2-a} $E$ is not $\sigma_2$-semistable;
    \item\label{s2-b} $E$ is strictly $\sigma_2$-semistable; 
    \item\label{s2-c} $E$ is $\sigma_2$-stable. 
\end{enumerate}
First suppose case~\eqref{s2-a} happens; then we show that $E \in P_{\sigma_2}(0, 1]$. If not, consider its HN filtration, and let $E_1$ be the first object of maximum phase and $E_n$ be the last object of minimum phase with respect to $\sigma_2$. Then the weak Mukai lemma \eqref{eq_MukaiLemma} applied to the triangle $E_1 \rightarrow E \rightarrow E/E_1$ and  Proposition~\ref{prop-PY}\eqref{min-hom1} imply $\hom^1(E_1, E_1) < \hom^1(E, E)$. Thus, by the induction assumption, $E_1$ is $\sigma_1$-semistable with $\phi_{\sigma_1}(E_1) = \phi_{\sigma_2}(E_1)$. The existence of a non-zero map $E_1 \rightarrow E$ and the stability of $E$ with respect to $\sigma_1$ imply that   
	\begin{equation}\label{phase-1}
	\phi_{\sigma_2}(E_1) = \phi_{\sigma_1}(E_1) \leq \phi_{\sigma_1}(E) \leq 1. 
	\end{equation}
   Applying the same argument to the last piece in the HN filtration $Q \rightarrow E \rightarrow E_n$, we get
   \begin{equation}\label{phase-2}
   0 < \phi_{\sigma_1}(E) \leq \phi_{\sigma_1}(E_n) = \phi_{\sigma_2}(E_n). 
   \end{equation}
   Therefore, $E \in P_{\sigma_2}(0, 1]$. Moreover, since the stability functions for $\sigma_1$ and $\sigma_2$ are the same, we obtain $\phi_{\sigma_1}(E) = \phi_{\sigma_2}(E)$; thus \eqref{phase-1} implies that $\phi_{\sigma_2}(E_1) \leq \phi_{\sigma_2}(E)$, which is not possible. Thus $E$ is $\sigma_2$-semistable. 
   
   Now suppose case~\eqref{s2-b} happens, so $E$ is strictly $\sigma_2$-semistable. First assume $k/r <2$. Let $\{E_i\}_{i \in I}$ be the Jordan--H\"older filtration of $E$ with respect to $\sigma_2$. By the weak Mukai lemma \eqref{eq_MukaiLemma} and the induction assumption, it follows that the $E_i$ are $\sigma_1$-stable of the same phase as with respect to $\sigma_2$. Thus $E$ is strictly $\sigma_1$-semistable of the same phase as with respect to $\sigma_2$. If $k = 2r =4$, we apply Lemma~\ref{lem-posibilities}. If $E$ is $S$-equivalent to $G^{\oplus k_1} \oplus S(G)[-2]^{\oplus k_2}$ for a $\sigma_2$-stable object $G$ and $k_1+k_2 >1$, then \eqref{hom} gives
   \begin{equation*}
      \hom^1(E, E) \geq -\chi(E, E) +1 = -(k_1+k_2)^2\chi(G, G) +1 > -\chi(G, G) +2 \geq \hom^1(G, G). 
   \end{equation*}
   
   Thus $G$ and $S(G)$ are $\sigma_1$-stable of the same phase. Hence, $E$ is strictly $\sigma_1$-semistable and $\phi_{\sigma_1}(E) = \phi_{\sigma_2}(E)$. If  case~\eqref{ii} of Lemma~\ref{lem-posibilities} happens, then the same argument as in the weak Mukai lemma implies that $\hom^1(E_i, E_i) < \hom^1(E, E)$; thus again the claim follows by the induction assumption.  
   
   Finally, assume case~\eqref{s2-c} happens. If $E$ is strictly $\sigma_1$-semistable, then the same argument as in case~\eqref{s2-b}  implies that $E$ is also strictly $\sigma_2$-semistable. Thus $E$ must be $\sigma_1$-stable, and we only need to show $E$ has the same phase with respect to both $\sigma_1$ and $\sigma_2$. 
   
   We know $E \in P_{\sigma_1}(0, 1]$. Since $E$ is $\sigma_2$-stable, there is an $m \in \Z$ such that $E[2m] \in P_{\sigma_2}(0, 1]$. If $m \neq 0$, then the same argument as in Step 1 implies that $\chi(Q_j, E) = 0$ for $j=1, 2$, and so $[E] = 0$ in $\mathcal{N}(\cT)$, which is not possible. Thus $ m=0$, and $E$ has the same phase with respect to both $\sigma_1$ and $\sigma_2$. 
\end{proof}

As suggested to us by Zhiyu Liu and Shizhuo Zhang, we can slightly relax the condition~\eqref{C3}.

\begin{Lem}\label{lem-relax}
	If\, $k= 2r =4$, we can change the condition~\eqref{C3} to the following:  there are three objects $Q_1, Q_2, Q_2' \in \cT$ satisfying \eqref{minimal} such that:  
	\begin{enumerate}
	    \item\label{lem-relax-a} $Q_2$ and $Q_2'$ have the same class in $\mathcal{N}(\cT)$; 
	    \item\label{lem-relax-b} $Q_1$ is not isomorphic to $Q_2$ or $Q_2'[1]$; 
	    \item\label{lem-relax-c} $\hom(Q_2, Q_1) \neq 0$ and $\hom(Q_1, Q_2'[1]) \neq 0$;  
	    \item\label{lem-relax-d}  $\hom(Q_2', Q_2[3]) = 0$. 
	\end{enumerate}
\end{Lem}

\begin{proof}
Let $\sigma = (Z, \cA)$ be an $S$-invariant stability condition on $\cT$. 
By Proposition~\ref{prop-PY}\eqref{minimal-stable}, we know all three objects $Q_1, Q_2, Q_2'$ are $\sigma$-stable.
Hence, up to $\widetilde{\text{GL}}_2^+(\R)$-action, we may assume $Q_1$ is of phase $1$ and $Z(Q_1) = -1$. If $\phi(Q_2) \in (0, 1]$, then we can assume $Z(Q_2) = i$ and proceed as before. So assume for a contradiction that $Q_2 \notin \cA$. 
 The assumptions~\eqref{lem-relax-b} and~\eqref{lem-relax-c} imply that 
\begin{equation*}
    \phi(Q_2) < \phi(Q_1) = 1. 
\end{equation*}
If $\phi(Q_2) \in (-\infty, -1)$, then $\phi(S(Q_2)) < 1$, and so $\hom(Q_1, S(Q_2)) = 0$, which is not possible by condition~\eqref{lem-relax-c}. This implies that 
\begin{equation}\label{f}
    -1 \leq \phi(Q_2) \leq 0. 
\end{equation}
On the other hand, conditions~\eqref{lem-relax-b} and~\eqref{lem-relax-c} give $1= \phi(Q_1) < \phi(Q_2') +1$. This implies that $1 \leq \phi(Q_2')$ because $Q_2$ and $Q_2'$ have the same class. Since $\hom(Q_2'[1], S(Q_1)) \neq 0$, we obtain 
\begin{equation*}
    1 \leq \phi(Q_2') \leq 2. 
\end{equation*}
Thus $\phi(Q_2') = \phi(Q_2) +2$, which implies  that 
\begin{equation*}
    \hom(Q_2', Q_2[i]) =  0 \quad \text{for $i \leq 1$ and $i \geq 5$}. 
\end{equation*}
Therefore, 
\begin{align*}
 -\hom(Q_2', Q_2[3]) + \hom(Q_2', Q_2[2]) + \hom(Q_2', Q_2[4]) =  \chi(Q_2', Q_2) = \chi(Q_2, Q_2) < 0, 
\end{align*}
and so $\hom(Q_2', Q_2[3]) \neq 0$, which is in contradiction to  assumption~\eqref{lem-relax-d}. 
\end{proof}

\begin{Rem} \label{rem_applicationto Fano3}
Theorem~\ref{thm_uniqueSinv} can be applied to prove the uniqueness of Serre-invariant stability conditions on the Kuznetsov component of certain Fano threefolds of Picard rank $1$ and index~$2$ (see~\cite[Theorem 2.3]{kuznetsov:fano-threefolds} for the classification in terms of the degree). More precisely, as explained in Corollary~\ref{cor-s-invariant-ku} and Remark~\ref{sec:quarticdoublesolid}, it applies to the Kuznetsov component of a cubic threefold (degree $3$ case) and to the quartic double solid (degree $2$ case), respectively. Theorem~\ref{thm_uniqueSinv} (via the relaxed condition in Lemma~\ref{lem-relax}) also applies to the Kuznetsov component of a Gushel--Mukai threefold, which is a Fano threefold of Picard rank $1$, index $1$ and genus~$6$, by~\cite[Corollary 4.5]{PR}. Recently, the uniqueness of Serre-invariant stability conditions for the above examples has been independently proved in~\cite[Theorem 4.25]{zhang-hochschild:} and~\cite[Theorem 4.24]{Zhang}. But we believe that the general criterion given in Theorem~\ref{thm_uniqueSinv} could be potentially applied in other interesting geometric examples. For instance, in the next section we explain an application to the Kuznetsov component of a very general cubic fourfold.
\end{Rem}

\subsection*{Aside: Very general cubic fourfolds} \label{subsec_cubicfourfolds}
If the Serre functor of $\cT$ is $S=[2]$, \textit{i.e.}\ $\cT$ is a $2$-Calabi-Yau category, then clearly any stability condition is Serre invariant. An example of such a category is the Kuznetsov component of a cubic  fourfold. In this section we apply Theorem~\ref{thm_uniqueSinv} to show that if $X$ is a very general cubic fourfold, then there is a unique $\widetilde{\text{GL}}^+(2, \mathbb{R})$-orbit of stability conditions on the Kuznetsov component of $X$.

Recall that the bounded derived category of a cubic fourfold $X$ has a semiorthogonal decomposition of the form
$$\Db(X)= \langle \ku(X), \cO_X, \cO_X(1), \cO_X(2) \rangle$$
by~\cite{Kuz_cubic4}, where $\ku(X)=\langle \cO_X, \cO_X(1), \cO_X(2) \rangle^{\perp}$. The Serre functor of $\ku(X)$ satisfies $\cS=[2]$. By~\cite{AT}, the numerical Grothendieck group of $X$ contains two classes $\lambda_1$ and $\lambda_2$ spanning a rank $2$ $A_2$-lattice 
$$\langle \lambda_1, \lambda_2 \rangle \cong
\begin{pmatrix}
\hphantom{{-}}2 & -1 \\
-1 & \hphantom{{-}}2
\end{pmatrix}$$
with respect to the intersection pairing $(-,-):=-\chi(-,-)$. We say that $X$ is very general if $\cN(\ku(X))=\langle \lambda_1, \lambda_2 \rangle$. 

\begin{Cor}
Let $X$ be a very general cubic fourfold. If $\sigma_1$, $\sigma_2$ are stability conditions on $\ku(X)$, then there exists a $\tilde{g} \in \widetilde{\emph{GL}}^+(2, \mathbb{R})$ such that $\sigma_1= \sigma_2 \cdot \tilde{g}$. 
\end{Cor}

\begin{proof}
The Serre functor of $\ku(X)$ satisfies $\cS^2=[4]$, as required in~\eqref{C1}.
Since $X$ is very general, we have
$$\cN(\ku(X)) \cong \langle \lambda_1, \lambda_2 \rangle \cong \begin{pmatrix}
\hphantom{{-}}2 & -1 \\
-1 & \hphantom{{-}}2
\end{pmatrix}.
$$
For every $v \in \cN(\ku(X))$, we have $v=a \lambda_1 + b \lambda_2$. Thus 
$$v^2=2 a^2 + 2b^2 -2ab \geq 2.$$
It follows that $\chi(v, v)=-v^2 \leq -2$, which implies that
\[\ell_{\ku(X)}:= \text{max}\lbrace \chi(v,v): v \neq 0 \in \cN(\ku(X)) \rbrace= -2 < 0,\] as required in~\eqref{C2}.

We now check~\eqref{C3}. Recall that given a line $\ell \subset X$, we can define the objects $F_\ell$, $P_\ell \in \ku(X)$ sitting in
$$0 \to F_\ell \to \cO_X^{\oplus 4} \to \cI_\ell(1) \to 0,$$
$$\cO_X(-1)[1] \to P_\ell \to \cI_\ell.$$
Their classes in $\cN(\ku(X))$ are $v(F_\ell)=\lambda_1$ and $v(P_\ell)=\lambda_1 + \lambda_2$ as computed in~\cite[Section 6.3]{paolo_twistedcubics}. By~\cite[Lemma A.5]{bayer:stability-conditions-kuznetsov-component}, they are stable with respect to every stability condition on $\ku(X)$, and by~\cite[(6.3.1)]{paolo_twistedcubics}, their phases with respect to the stability conditions constructed in~\cite{bayer:stability-conditions-kuznetsov-component} satisfy
\begin{equation} \label{eq_phases}
\phi(F_\ell) < \phi(P_\ell) < \phi(F_\ell)+1.    
\end{equation}
Since $\chi(F_\ell, F_\ell)=-2=2- \hom^1(F_\ell, F_\ell)$ and similarly for $P_\ell$, we have
$$3 \leq \hom^1(F_\ell,F_\ell)=\hom^1(P_\ell,P_\ell)=4 < 6.$$
For two lines $\ell$ and $\ell'$, we have
\begin{align*}
\hom(F_\ell, P_{\ell'}) &= \hom(\cI_\ell(1), P_{\ell'}[1])\\
&=\hom(\cI_\ell(1), \cI_{\ell'}[1])\\
&=\hom(\cI_\ell(1), \cO_{\ell'})\\
&=\hom(\cO_\ell(1), \cO_{\ell'}[1]).
\end{align*}
The latter is equal to $0$ if $\ell \cap \ell' = \emptyset$. Assume $\ell$ and $\ell'$ intersect in a point. Then by the local-to-global spectral sequence, we only have to consider $H^0(\mathcal{E}xt^1(\cO_\ell(1), \cO_{\ell'}))$. Since $\mathcal{E}xt^1(\cO_\ell(1), \cO_{\ell'})$ is supported on the intersection point, we conclude that \smash{$\hom(\cO_\ell(1), \cO_{\ell'}[1])=1$} and thus $\hom(F_\ell, P_{\ell'}) =1 \neq 0$ if $\ell$ and $\ell'$ intersect in a point.

We know $\chi(P_{\ell'}, F_\ell)=-(\lambda_1+\lambda_2, \lambda_1)=-1$. By \eqref{eq_phases} and Serre duality,  $\hom(P_{\ell'}, F_\ell[i]) = 0$ for $i \neq 0, 1, 2$; hence
$$-1 =\chi(P_{\ell'}, F_\ell)= \hom(P_{\ell'}, F_\ell)- \hom^1(P_{\ell'}, F_\ell)+\hom^2(P_{\ell'}, F_\ell),$$
which implies $\hom(P_{\ell'}, F_\ell[1]) \geq 1$ for every pair of lines $\ell$, $\ell'$. Thus~\eqref{C3} holds if we set $Q_1 = P_{\ell'}$ and $Q_2 = F_{\ell}$.
\end{proof}

\section{Kuznetsov component of cubic threefolds} \label{sec_kuzcubicthreefold}
From now on, we assume $X$ is a smooth complex
cubic threefold, and $\cO_X(H)$ denotes the corresponding very ample line bundle. 
\subsection{Kuznetsov component}
The Kuznetsov component $\ku(X)$ is the right-orthogonal complement of the exceptional collection $\cO_X, \cO_X(H)$ in $\Db(X)$ 
sitting in the semiorthogonal decomposition
$$\Db(X)= \langle \ku(X), \cO_X, \cO_X(H) \rangle$$
(see~\cite{kuznetsov-derived-category-cubic-3folds-V14}). By~\cite[Proposition  2.7]{macri:categorical-invarinat-cubic-threefolds}, the numerical Grothendieck group $\mathcal{N}(\ku(X))$ of $\ku(X)$ is a rank $2$ lattice
\begin{equation*}
  \mathcal{N}(\ku(X)) = \Z\, [\cI_{\ell}] \oplus \Z\, [\cS(\cI_{\ell})],
\end{equation*}
where $\cI_{\ell}$ is the ideal sheaf of a line $\ell$ in $X$ and $\mathcal{S}$ denotes the Serre functor of $\ku(X)$. With respect to this basis, the Euler form $\chi\_{\ku(X)}(-, -)$ on $\mathcal{N}(\ku(X))$ is represented by
\begin{equation}
\label{eq_eulerform}
\begin{pmatrix}
 -1 & -1\\\hphantom{{-}}0 & -1   
\end{pmatrix}.
\end{equation}
Note that by~\cite[Lemma 4.1, Lemma 4.2]{kuznetsov-derived-category-cubic-3folds-V14}, the functor
$$\mathrm{O} \colon \Db(X) \to \Db(X), \quad \mathrm{O}(-)=L_{\cO_X}(- \otimes \cO_X(H))[-1]$$
satisfies $\mathrm{O}|_{\ku(X)}^2 = \cS^{-1}[1]$ and  $\mathrm{O}|_{\ku(X)}^3 \cong [-1]$. Thus 
	\begin{equation*}
	\cS[-3] = \cS^{-2}[2] = \mathrm{O}|_{\ku(X)}^4 = \mathrm{O}|_{\ku(X)}[-1],
	\end{equation*}	
	which implies that 
	\begin{equation*}
	\cS = \mathrm{O}|_{\ku(X)}[2] =  L_{\cO_X}(- \otimes \cO_X(H))[1]. 
	\end{equation*}

\subsection{Stability conditions on $\ku(X)$} \label{subsection-kuz-stability}

Stability conditions on $\ku(X)$ have been first constructed in~\cite{macri:categorical-invarinat-cubic-threefolds} and more recently in~\cite{bayer:stability-conditions-kuznetsov-component}. We recall here the definition of the latter, while the former will be reviewed in Section~\ref{sec-inducedstabcond}.

For $\beta \in \R$, denote by $\Coh^\beta(X)$ the heart of a bounded t-structure obtained by tilting $\Coh(X)$ with respect to slope stability at slope $\mu_H=\beta$. By~\cite[Proposition 2.12]{bayer:stability-conditions-kuznetsov-component}, for $\alpha \in \R_{>0}$, the pair $\sigma_{\alpha,\beta}=(\Coh^\beta(X), Z_{\alpha,\beta})$ defines a weak stability condition on $\Db(X)$, known as tilt stability, with respect to the lattice $\Lambda \cong \Z^3$ generated by the elements of the form $(H^3\ch_0(E), H^2\ch_1(E), H\ch_2(E)) \in \Q^3$ for $E \in \Db(X)$ (see Section~\ref{sec-tiltstab}). Here 
$$Z_{\alpha, \beta}
\coloneqq \frac{1}{2}\alpha^2 H^{3}\ch_{0}^{\beta}-H\ch_{2}^{\beta}
+i H^{2}\ch_{1}^{\beta}$$
with $\ch^\beta \coloneqq e^{-\beta}\ch=(\ch_0, \ch_1-\beta H \ch_0, \ch_2 -\beta H \ch_1 +\frac{1}{2}\beta^2 H^2\ch_0, \dots)$. Its associated slope is
$$\mu_{\alpha,\beta}(E)=-\frac{\Re Z_{\alpha,\beta}(E)}{\Im Z_{\alpha,\beta}(E)} \quad \text{for }\Im Z_{\alpha,\beta}(E) \neq 0,
$$
where $\Re[-]$ and $\Im[-]$ are the real and imaginary parts. 

Now consider the tilted heart $\Coh^0_{\alpha, \beta}(X) = \langle \mathcal{F}_{\alpha, \beta}[1] , \cT_{\alpha, \beta} \rangle$, where $\mathcal{F}_{\alpha, \beta}$ (resp.\ $\cT_{\alpha, \beta}$) is the subcategory of objects in $\Coh^{\beta}(X)$ with $\mu^+_{\alpha, \beta} \leq 0$ (resp.\ $\mu^-_{\alpha, \beta} > 0$). 
By~\cite[Proposition~2.14]{bayer:stability-conditions-kuznetsov-component}, the pair $\sigma^0_{\alpha, \beta} = \left(\Coh^0_{\alpha, \beta}(X) , -iZ_{\alpha, \beta}\right)$ is a weak stability condition on $\Db(X)$. 

Sheaves supported in codimension 3 are the only objects in the heart $\Coh^{\beta}(X)$ whose central charge $Z_{\alpha, \beta}$ vanishes. We denote by $\Coh_0(X)$ the category of sheaves on $X$ of codimension 3.    

\begin{Prop}\label{prop-rotated-stability}
Any $\sigma^0_{\alpha, \beta}$-$($semi$\,)$stable object $E \in \Coh^0_{\alpha, \beta}(X)$ is $\sigma_{\alpha, \beta}$-$($semi$\,)$stable if it does not lie in an exact triangle of the form
\begin{equation*}
F[1] \rightarrow E \rightarrow T,     
\end{equation*}
where $F \in \mathcal{F}_{\alpha, \beta}$ and $T \in \Coh_0(X)$. Conversely, take a $\sigma_{\alpha, \beta}$-$($semi$\,)$stable object $E$ such that either
\begin{enumerate}
    \item\label{prop-rot-stab-a} $E \in \cT_{\alpha, \beta}$ and $\Hom(\Coh_0(X), E) = 0$, or
    \item\label{prop-rot-stab-b} $E \in \mathcal{F}_{\alpha, \beta}$ and $\Hom(\Coh_0(X), E[1]) = 0$. 
\end{enumerate}
Then $E$ is $\sigma^0_{\alpha, \beta}$-$($semi$\,)$stable.   
\end{Prop}

\begin{proof}
First assume $E \in \Coh^0_{\alpha, \beta}$ is $\sigma^0_{\alpha, \beta}$-(semi)stable, so by definition it lies in an exact triangle 
\begin{equation}\label{exact}
    F[1] \rightarrow E \rightarrow T
\end{equation}
for some $F \in \mathcal{F}_{\alpha, \beta}$ and $T \in \mathcal{T}_{\alpha, \beta}$. If $Z_{\alpha, \beta}(T) \neq 0$, \textit{i.e.}\ $T \notin \Coh_0(X)$, we have 
\begin{equation*}
    \Re[-iZ_{\alpha, \beta}(T)] \geq 0 \quad \text{and} \quad \Re[-iZ_{\alpha, \beta}(F[1])] < 0. 
\end{equation*}
This shows that the phase of $F[1]$ is bigger than that of $T$ with respect to $\sigma^0_{\alpha, \beta}$. Thus the exact triangle \eqref{exact} implies that one of the following cases happens: 
\begin{enumerate}[label=(\roman*),ref=\roman*]
    \item\label{p1}  $E = T \in \mathcal{T}_{\alpha, \beta}$; 
    \item\label{p2} $E = F[1] \in \mathcal{F}_{\alpha, \beta}[1]$; or 
    \item\label{p3} $T \in \Coh_0(X)$. 
\end{enumerate}
Case~\eqref{p3} is excluded in the statement, so we may assume the first two cases happen. We only consider the case $E \in \mathcal{T}_{\alpha, \beta}$; the other one can be shown by a similar argument. Suppose for a contradiction that $E$ is not $\sigma_{\alpha, \beta}$-(semi)stable, and let 
\begin{equation}\label{exact-2}
E_1 \rightarrow E \rightarrow E_2    
\end{equation}
be a destabilising sequence in $\Coh^{\beta}(X)$. We may assume $E_1$ is $\sigma_{\alpha, \beta}$-semistable; thus  
\begin{equation*}
    \mu_{\alpha, \beta}^{-}(E_1) =\mu_{\alpha, \beta}(E_1) \geq \mu_{\alpha, \beta}(E_2) \geq \mu_{\alpha, \beta}^{-}(E_2) \geq \mu_{\alpha, \beta}^{-}(E)  \geq 0. 
\end{equation*}
The right inequality follows from the definition of $\mathcal{T}_{\alpha, \beta}$. Thus both $E_1$ and $E_2$ lie in $\mathcal{T}_{\alpha, \beta}$; hence \eqref{exact-2} is also a destabilising sequence for $E$ in $\Coh^0_{\alpha, \beta}$ with respect to $\sigma^0_{\alpha, \beta}$, which leads to a contradiction.   

\bigskip

Now assume $E \in \mathcal{T}_{\alpha, \beta}$ is $\sigma_{\alpha, \beta}$-(semi)stable and $\Hom(\Coh_0(X), E) = 0$. In the rest of the argument, all $T_i$ lie in $\mathcal{T}_{\alpha, \beta}$ and all $F_i$ lie in $\mathcal{F}_{\alpha, \beta}$. Suppose $E_1 \rightarrow E \rightarrow E_2$ is a destabilising sequence with respect to $\sigma^0_{\alpha, \beta}$. We can assume $E_2$ is $\sigma^0_{\alpha, \beta}$-semistable. Thus by the above argument, we have either \begin{inparaenum}\item $E_2 \in \mathcal{F}_{\alpha, \beta}[1]$, or \item $E_2 \in \mathcal{T}_{\alpha, \beta}$, or \item $E_2$ lies in an exact triangle $F_2[1] \rightarrow E_2 \rightarrow T_2$, where $T_2 \in \Coh_0(X)$. \end{inparaenum}

Since $E \in \mathcal{T}_{\alpha, \beta}$, the phase of $E$ with respect to $\sigma^0_{\alpha, \beta}$ is less than or equal to $\frac{1}{2}$. But we know the phase of objects in  $\mathcal{F}_{\alpha, \beta}[1]$ is bigger than $\frac{1}{2}$; thus case~\eqref{p1} cannot happen. In  case~\eqref{p3} we know the phase of $E_2$ is equal to the phase of $F_2[1]$, so again this case cannot happen. Hence we may assume $E_2 \in \mathcal{T}_{\alpha, \beta}$. 

By definition, we know $E_1$ lies in an exact triangle $F_1[1] \rightarrow E_1 \rightarrow T_1$. But the composition of the injections $F_1[1] \rightarrow E_1 \rightarrow E$ is zero because $\hom(\mathcal{F}_{\alpha, \beta}[1], \mathcal{T}_{\alpha, \beta}) = 0$. Thus $E_1 \in \mathcal{T}_{\alpha, \beta}$, and so the destabilising sequence is an exact sequence in the original heart $\Coh^{\beta}(X)$ which (semi)destabilises $E$ with respect to $\sigma_{\alpha, \beta}$, leading to a contradiction. 

\bigskip

Finally, suppose $E \in \mathcal{F}_{\alpha, \beta}$ is $\sigma_{\alpha, \beta}$-(semi)stable and we have $\hom(\Coh_0(X), E[1]) = 0$. Let $E_1 \rightarrow E[1] \rightarrow E_2$ be a destabilising sequence with respect to $\sigma^0_{\alpha, \beta}$. We may assume $E_1$ is \mbox{$\sigma^0_{\alpha, \beta}$-semistable.} We know $E_1 \notin \Coh_0(X)$, and the phase of objects in $\mathcal{T}_{\alpha, \beta} \setminus \Coh_0(X)$ with respect to $\sigma_{\alpha, \beta}^0$ is less than the phase of objects in $\mathcal{F}_{\alpha, \beta}[1]$. Thus the same argument as in case~\eqref{prop-rot-stab-a} implies that $E_1$ lies in an exact triangle $F_1[1] \rightarrow E_1 \rightarrow T_1$, where $T_1 \in \Coh_0(X)$ and $F_1 \neq 0$. By definition, $E_2$ lies in an exact triangle $F_2[1] \rightarrow E_2 \rightarrow T_2$. Since $\hom(\mathcal{F}_{\alpha, \beta}[1] , \mathcal{T}_{\alpha, \beta}) = 0$, the composition of surjections $E[1] \rightarrow E_2 \rightarrow T_2$ is zero. Thus $T_2 = 0$ and $E_2 = F_2[1]$. Taking cohomology from the destabilising sequence with respect to the heart $\Coh^{\beta}(X)$ gives the long exact sequence 
\begin{equation*}
    0 \rightarrow F_1 \rightarrow E \rightarrow F_2 \rightarrow T_1 \rightarrow 0.
\end{equation*}
Thus $F_1$ is a subobject of $E$ in $\Coh^{\beta}(X)$. We know the phase of $E_1$ is equal to the phase of $F_1[1]$ with respect to $\sigma^0_{\alpha, \beta}$, and it is bigger than or equal to the phase of $E[1]$. This implies that the phase of $F_1$ is bigger than or equal to the phase of $E$ with respect to $\sigma_{\alpha, \beta
}$, leading to a contradiction. 
\end{proof}

By restricting weak stability conditions $\sigma^0_{\alpha, \beta}$ to the Kuznetsov component $\ku(X)$, we obtain a stability condition on it.   

\begin{Thm}[\protect{\cite[Theorem 6.8]{bayer:stability-conditions-kuznetsov-component}, \cite[Theorem 3.3 and Corollary 5.5]{pertusi:some-remarks-fano-threefolds-index-two}}]
\label{thm_stabcondinduced}
For every $(\alpha,\beta)$ in the set 
\begin{equation}\label{setV}
    V \coloneqq \left\{ (\alpha,\beta) \in \R_{>0} \times \R : -\frac{1}{2} \leq \beta < 0, \alpha < -\beta, \text{ or } -1 < \beta <-\frac{1}{2}, \alpha \leq 1+\beta \right\},
\end{equation}
the pair $\sigma(\alpha,\beta)=(\cA(\alpha,\beta), Z(\alpha,\beta))$ is a $\cS$-invariant stability condition on $\ku(X)$, where
\begin{equation*}
    \cA(\alpha, \beta):=\Coh_{\alpha, \beta}^{0}(X)\cap \ku(X) \quad \text{and} \quad
    Z(\alpha, \beta) \coloneqq -i Z_{\alpha, \beta}|_{\ku(X)}.
\end{equation*}
\end{Thm}

As an application of Theorem~\ref{thm_uniqueSinv}, we get the following.

\begin{Cor}\label{cor-s-invariant-ku}
All $\cS$-invariant stability conditions on $\ku(X)$ such as $\sigma(\alpha, \beta)$ for $(\alpha, \beta) \in V$ lie in the same orbit with respect to the action of $\widetilde{\emph{GL}}^+(2, \mathbb{R})$. 
\end{Cor}

\begin{proof}
We know $\cS^3 = [5]$ and that the numerical Grothendieck group $\mathcal{N}(\ku(X))$ is a lattice of rank $2$ such that $\chi\_{\ku(X)}(v, v) \leq -1$ for any $0 \neq v \in \mathcal{N}(\ku(X))$. Thus  conditions~\eqref{C1} and~\eqref{C2} hold for $\ku(X)$. We also have $ \dim_{\mathbb{C}} \Hom^1(\cI_{\ell}, \cI_{\ell}) = 2$ as the Hilbert scheme of lines in $X$ is a smooth surface; see \cite[Section 1]{altman}. Thus the claim follows from Theorem~\ref{thm_uniqueSinv}.  
\end{proof}

\begin{Rem}\label{sec:quarticdoublesolid}
Theorem~\ref{thm_uniqueSinv} also applies to the Kuznetsov component of the quartic double solid $Y$, which is a double cover $Y \to \PP^3$ ramified in a quartic. In fact, by~\cite{kuznetsov:fano-threefolds} there is a semiorthogonal decomposition of the form
$\Db(Y)= \langle \ku(Y), \cO_Y, \cO_Y(H) \rangle$,
where $H$ is an ample class, and the numerical Grothendieck group of $\ku(Y)$ is a  rank $2$ lattice represented by the matrix
\begin{equation*}
\begin{pmatrix}
 -1 & -1\\-1 & -2  
\end{pmatrix}.    
\end{equation*}
An easy computation shows that~\eqref{C2} holds. Moreover, the Serre functor of $\ku(Y)$ satisfies $S_{\ku(Y)}= \iota[2]$, where $\iota$ is an involutive autoequivalence of $\ku(Y)$, by~\cite[Corollary 4.6]{Kuz:calabi}. Thus~\eqref{C1} holds with $k=2r=4$. It remains to find two objects as in~\eqref{C3}. Consider the ideal sheaf of a line $\cI_\ell$ and its derived dual $\cJ_\ell:=R\Hom(\cI_\ell, \cO(-H)[1])$. By~\cite[Remark 3.8]{pertusi:some-remarks-fano-threefolds-index-two}, for proper choices of lines, they provide the required objects.
\end{Rem}

\subsection{Stable objects in \texorpdfstring{$\boldsymbol{\ku(X)}$}{Ku(X)} with minimal \texorpdfstring{$\boldsymbol{\hom^1}$}{hom\textasciicircum 1}} \label{section_minimalhom1}
By \eqref{eq_eulerform}, the classes in $\mathcal{N}(\ku(X))$ with self-intersection $-1$ are 
$$\pm [\cI_{\ell}]\, , \quad \pm [\cS(\cI_{\ell})]\ \quad \text{and} \quad \pm [\cS^2(\cI_{\ell})] = \pm \left( [\cI_{\ell}] - [\cS(\cI_{\ell})] \right).$$
In this section we investigate the moduli of semistable objects of the above classes. We denote by $M^{\sigma}_{\ku(X)}(v)$ the moduli space of $\sigma$-semistable objects of class $v \in \mathcal{N}(\ku(X))$.   

\begin{Thm}\label{thm-moduli-spaces}
Let $\sigma$ be an $\cS$-invariant stability condition on $\ku(X)$.  
\begin{enumerate}
    \item\label{thm-mod-sp-a} The moduli space $M^{\sigma}_{\ku(X)}([\cI_{\ell}])$ is isomorphic to the moduli space $M_X(v)$ of slope-stable sheaves on $X$ with Chern character $v \coloneqq \ch(\cI_{\ell}) = \left(1, 0, -\frac{1}{3}H^2, 0\right)$.
    \item\label{thm-mod-sp-b} The moduli space $M^{\sigma}_{\ku(X)}([\cS(\cI_{\ell})])$ is isomorphic to the moduli space $M_X(w)$ of slope-stable sheaves on $X$ with Chern character $w \coloneqq \left(2, -H, -\frac{1}{6}H^2, \frac{1}{6}H^3\right)$.
    \item\label{thm-mod-sp-c} The moduli space $M^{\sigma}_{\ku(X)}([\cS^2(\cI_{\ell})])$ is isomorphic to the moduli space $M_X(v-w)$ of  large volume limit stable complexes of character $v-w$. 
\end{enumerate}
\end{Thm}

In particular, all the  three moduli spaces above are isomorphic to the Fano variety $\Sigma(X)$ of lines in $X$. This proves Theorem~\ref{thm-introduction-2}. Part~\eqref{thm-mod-sp-a} of Theorem~\ref{thm-moduli-spaces} is proved in~\cite{pertusi:some-remarks-fano-threefolds-index-two}; however, we prove it here via a slightly different argument. 

\bigskip

We start with a slope-stable sheaf $E$ with Chern character $v$. So $E$ is torsion-free with double dual $E^{\vee\vee} = \cO_X$. Thus it lies in the exact sequence 
\begin{equation*}
    E \hookrightarrow \cO_X \twoheadrightarrow Q, 
\end{equation*}
where $Q$ is a torsion sheaf of character $\left(0, 0, \frac{1}{3}H^2 , 0\right)$, which is therefore isomorphic to the structure sheaf $\cO_{\ell}$ of a line $\ell$ in $X$. Thus $E$ is isomorphic to the ideal sheaf $\cI_{\ell}$. 

By Lemma~\ref{lem-large-volume}, $\cI_{\ell}$ is $\sigma_{\alpha, \beta}$-stable for $\beta <0$ and $\alpha \gg 0$. Applying~\cite[Lemma 3.5]{feyz:slope-stability-of-restriction} to $\cI_{\ell}$, we see that there is no wall for it crossing the vertical line $\beta = \beta_0 = -1/{H^3}$. Thus $\cI_{\ell}$ is $\sigma_{\alpha, \beta_0}$-stable for any $\alpha > 0$.
Since $\mu_{\alpha, \beta_0}(\cI_{\ell}) <0$, we have $\cI_{\ell}[1] \in \Coh_{\alpha, \beta_0}^0(X)$. For any sheaf $Q \in \Coh_{0}(X)$, we know 
\begin{equation*}
    \hom(Q, \cO_{\ell}) = 0 = \hom(Q, \cO_X[1]). 
\end{equation*}
Thus $\hom(Q, \cI_{\ell}[1]) = 0$, and so Proposition~\ref{prop-rotated-stability} implies that $\cI_{\ell}[1]$ is $\sigma^0_{\alpha, \beta_0}$-stable. Since $\cI_{\ell} \in \ku(X)$, we immediately get that $\cI_{\ell}$ is $\sigma(\alpha, \beta_0)$-stable. Thus by Corollary~\ref{cor-s-invariant-ku}, we get the following. 

\begin{Lem}\label{lem-s-v}
The ideal sheaf\, $\cI_{\ell}$ for any line $\ell$ in $X$ is stable with respect to any $\cS$-invariant stability condition on $\ku(X)$. 
\end{Lem}

By definition, $\cS(\cI_{\ell})$ is also stable with respect to any $\cS$-invariant stability condition on $\ku(X)$. The next step is to describe the complex $\cS(\cI_{\ell})$. We know $\cI_{\ell}(H)$ is of character	$\left(1, H, \frac{1}{6}H^2, -\frac{1}{6}H^3\right)$, so $\chi(\cO_X, \cI_{\ell}(H)) = 3$. Consider the short exact sequence 
\begin{equation*}
	\cI_{\ell}(H) \hookrightarrow \cO_X(H) \twoheadrightarrow \cO_{\ell}(H). 
\end{equation*} 
Since the map $\Hom(\cO_X, \cO_X(H)) \rightarrow \Hom(\cO_X, \cO_{\ell}(H))$ is surjective 
and $H^i(\cO_X(H)) = 0$ for $i >0$ by the Kodaira vanishing theorem, we get 
\begin{equation*}
 \hom(\cO_X, \cI_{\ell}(H)[1]) = \hom(\cO_X, \cI_{\ell}(H)[2]) = 0.   
\end{equation*}
Thus we obtain
\begin{equation*}
	\cS(\cI_{\ell})[-1] = L_{\cO_X}(\cI_{\ell}(H)) = \mathrm{cone}(\cO_X^{\oplus 3} \xrightarrow{\ev} \cI_{\ell}(H)). 
\end{equation*}
Since Pic$(X) = \Z.H$, Lemma 3.5 of \cite{feyz:slope-stability-of-restriction} implies that $\cI_{\ell}(H)$ is $\sigma_{\alpha, \beta = 0}$-stable for any $\alpha >0$. Thus $\cI_{\ell}(H)$ is $\sigma_{\alpha, \beta}$-semistable along the numerical wall $W(\cI_{\ell}(H), \cO_X[1])$ where $\cI_{\ell}(H)$ and $\cO_X[1]$ have the same phase. Therefore, their extension $L_{\cO_X}(\cI_{\ell}(H))$ is also $\sigma_{\alpha, \beta}$-semistable along the wall. One can easily show that $W(\cI_{\ell}(H), \cO_X[1])$ intersects the horizontal line $\alpha = 0$ at two points with $\beta$-values $0$ and $\frac{1}{3}$. Thus by the definition of the heart $\Coh^{\beta}(X)$, $\cS(\cI_{\ell})$ is a two-term complex with cohomology in degree -2 and -1, and 
\begin{equation}\label{slope}
    \mu_H^+\left(\cH^{-2}(\cS(\cI_{\ell})   )\right) \leq 0.
\end{equation}
Also, since $\cH^{-1}(\cS(\cI_{\ell}))$ is a quotient sheaf of $\cI_{\ell}(H)$, we get 
\begin{equation*}
    1 \leq \mu_H\left(\cH^{-1}(\cS(\cI_{\ell})) \right).
\end{equation*}
We know $\ch\left(\cS(\cI_{\ell})\right) = w = \left(2, -H, -\frac{1}{6}H^2, \frac{1}{6}H^3\right)$. The rank of $\cH^{-1}(\cS(\cI_{\ell}))$ is lower than $1 = \rk(\cI_{\ell}(H))$. If the rank is $1$, then the image im$(\ev)$ in the category of coherent sheaves is of rank zero, which is not possible because $\cI_{\ell}(H)$ is torsion-free. Thus $\cH^{-1}(\cS(\cI_{\ell}))$ is of rank zero, so 
\begin{equation*}
    (-1)^{i+1}\ch_1\left( \cH^{i}(\cS(\cI_{\ell}))   \right).H^2 \geq 0
\end{equation*}
for $i=-2, -1$. Thus $ \ch_1\left(\cH^{-1}(\cS(\cI_{\ell})) \right)$ is  equal to either zero or $H$. In the latter case, we will have $\ch_{\leq 2}\left( \cH^{-2}(\cS(\cI_{\ell})) \right) = (2, 0, -\alpha)$ for some $\alpha \geq 0$ because of the Gieseker semistability of $\cO_X^{\oplus 3}$. Then im$(\ev)$ is a subsheaf of $\cI_{\ell}(H)$ of rank $1$, so it is torsion-free of class $\ch(\text{im}(\ev)) = (1, 0, \alpha)$ and so $\alpha \leq 0$; thus in total we obtain $\alpha =0$. 

Applying a similar argument also implies that $\ch_{3}\left( \cH^{-2}(\cS(\cI_{\ell})) \right) = 0$; thus $\cH^{-2}(\cS(\cI_{\ell})) = \cO_X^{\oplus 2}$, which is not possible by the definition of the evaluation map $\ev$. Therefore, $\cH^{-1}(\cS(\cI_{\ell}))$ is a sheaf supported in dimension at most $1$. 

We also claim $\cH^{-2}(\cS(\cI_{\ell}))$ is slope stable. Otherwise, there is a subsheaf $F$ of bigger slope. By \eqref{slope}, it must be of class $\ch_{\leq 1}(F) = (1, 0)$. The semistability of $\cS(\cI_{\ell})[-1]$ along the wall implies that $\cH^{-2}(\cS(\cI_{\ell}))$ is a reflexive sheaf. If not, $\cH^{-2}(\cS(\cI_{\ell}))[1]$, and so $\cS(\cI_{\ell})[-1]$, has a subobject supported in dimension at most $1$ of phase $1$, which is not possible. Since $\cH^{-2}(\cS(\cI_{\ell}))$ is of rank $2$, its quotient sheaf $\cH^{-2}(\cS(\cI_{\ell}))/F$ is stable, and it must  also be reflexive. Taking the double dual shows that $F^{\vee\vee} = F$; thus $F$ is a line bundle, and so $F \cong \cO_X$, which is in contradiction to the definition of the evaluation map $\ev$. Therefore, $\cH^{-2}(\cS(\cI_{\ell}))$ is slope stable. 

The next proposition implies that $\cH^{-1}(\cS(\cI_{\ell}))$ is zero and so the evaluation map $\ev$ is surjective and the kernel is a reflexive sheaf of Chern character $w$.  

\begin{Lem}\label{lem-reflexive}
		Take an $H$-slope-stable sheaf $E$ of Chern character $(2, -H, \ch_2, \ch_3)$. Then $\ch_2.H \leq -\frac{1}{6}H^3$, and if $\ch_2.H = -\frac{1}{6}H^3$, then $\ch_3 \leq \frac{1}{6}H^3$. This, in particular, implies that any $H$-slope-stable sheaf of Chern character $w$ is a reflexive sheaf. 
	\end{Lem}

\begin{proof}
    Proposition~\ref{prop-li-ch2} implies that $\ch_2.H \leq 0$. Since $c_2(E) = \frac{1}{2}H^2 - \ch_2$ is an integral class, we obtain $\ch_2.H \leq -\frac{1}{6}H^3$. Assume $\ch_2.H = -\frac{1}{6}H^3$. There is no wall for $E$ crossing the vertical line $\beta = -1$ by~\cite[Lemma 3.5]{feyz:slope-stability-of-restriction}; thus $E$ is $\sigma_{\alpha, \beta = -1}$-stable for any $\alpha >0$. One can easily compute that
      \begin{equation*}
      \lim_{\alpha \rightarrow 0 } \mu_{\alpha,\, \beta = -1}(\mathcal{O}_X(-2H)[1]) < \lim_{\alpha \rightarrow 0 } \mu_{\alpha,\, \beta = -1}(E), 
      \end{equation*}
      which implies that $\hom(E, \mathcal{O}_X(-2H)[1]) = 0$ because $\mathcal{O}_X(-2H)[1]$ is $\sigma_{\alpha, \beta}$-stable for any $\alpha >0$ and $\beta \in \R$; see \cite[Corollary 3.11]{bayer:the-space-of-stability-conditions-on-abelian-threefolds}. Moreover, since 
      \begin{equation*}
      \mu_H(E) = \frac{-1}{2} < \mu_H(\mathcal{O}_X), 
      \end{equation*}
      we get $\hom(\mathcal{O}_X, E) = 0$. Therefore, $\chi(\cO_X, E) = \ch_3(E) - \frac{1}{6}H^3 \leq 0$, as claimed.
\end{proof}

 Thus we obtain 
 	\begin{equation*}
 	\cS(\cI_{\ell}) = L_{\cO_X}(\cI_{\ell}(H))[1] = K_{\ell}[2]
 \end{equation*}
 for a slope-stable reflexive sheaf $K_{\ell}$ of Chern character $w$. The next proposition shows that any slope-stable sheaf $K$ of class $w$ is of the form $L_{\cO_X}(\cI_{\ell}(H))[-2]$ for a line $\ell$ on $X$. 

 \begin{Prop} \label{prop_slopestableclassw}
 Take a slope-stable sheaf $K$ of Chern character $w$. There exists a line $\ell$ on $X$ such that $K$ lies in the short exact sequence
 \begin{equation*}
     0 \rightarrow K \rightarrow \cO_X^{\oplus 3} \rightarrow \cI_{\ell}(H) \rightarrow 0 .
 \end{equation*}
 In other words, $K = \cS(\cI_{\ell})[-2]$. 
 \end{Prop}

 \begin{proof}
  Lemma~\ref{lem-reflexive} implies that $K$ is a reflexive sheaf; thus $K[1]$ is $\sigma_{\alpha, \beta}$-stable for $\beta \geq -\frac{1}{2}$ and $\alpha \gg 0$; see \cite[Proposition 4.17]{feyz:desing}. Since there is no wall for $K[1]$ crossing the vertical line $\beta = 0$, it is $\sigma_{\alpha, \beta}$-semistable along the numerical wall $W(K[1] , \cO_X[1])$. We have 
     \begin{equation*}
       \lim_{\alpha \rightarrow 0 } \mu_{\alpha,\, \beta = 0}(K[1]) < \lim_{\alpha \rightarrow 0 } \mu_{\alpha,\, \beta = 0}(\mathcal{O}_X(2H)).
      \end{equation*}
 Thus $\hom(K, \cO_X[2]) = \hom(\cO_X(2H) , K[1]) = 0$. Also, \smash{$\hom(K, \cO_X[3]) \!=\! \hom(\cO_X(2H) , K) \!=\! 0$} because $\mu_H(K) <0$. Therefore, 
 \begin{equation*}
    3=\chi(K, \cO_X) =  \hom(K, \cO_X) - \hom(K, \cO_X[1]), 
 \end{equation*}
which implies that there is a non-zero map $K[1] \rightarrow \cO_X[1]$. Therefore, $K[1]$ is unstable below the numerical wall $W(K[1] , \cO_X[1])$. 

Choose three linearly independent sections of $K^{\vee}$:
\begin{equation*}
    \psi \colon K[1] \rightarrow \cO_X^{\oplus 3}[1]. 
\end{equation*}
We know $K[1]$ and $\cO_X[1]$ are $\sigma_{\alpha, \beta}$-semistable of the same phase $\phi$ and $\cO_X[1]$ is a simple object in $\mathcal{P}(\phi)$. Therefore, the map $\psi$ is surjective in $\mathcal{P}(\phi)$. The kernel $\ker \psi$ in $\mathcal{P}(\phi)$ is of Chern character $\left(1, H, \frac{1}{6}H^2 , -\frac{1}{6}H^3\right)$. Taking cohomology gives the long exact sequence of coherent sheaves 
\begin{equation*}
    0 \rightarrow \cH^{-1}(\ker \psi) \rightarrow K \rightarrow \cO_X^{\oplus 3} \rightarrow \cH^0(\ker \psi) \rightarrow 0. 
\end{equation*}
The wall $W(K[1] , \cO_X[1])$ intersects the horizontal line $\alpha =0$ at two points with $\beta =0, \frac{1}{3}$. Thus moving along the wall implies that 
\begin{equation}\label{slope-a}
\frac{1}{6} \leq \mu_H^{-}\left(\cH^0(\ker \psi)\right). 
\end{equation}
Moreover, the slope stability of $K$ implies that 
\begin{equation}\label{slope-b}
    \mu_H\left(\cH^{-1}(\ker \psi)\right) \leq -\frac{1}{2}.
\end{equation}
Therefore, $0 < (-1)^{i}\ch_1\left(\cH^i(\ker \psi )\right).H$ and $\ch_1\left(\cH^0(\ker \psi)\right) - \ch_1\left(\cH^{-1}(\ker \psi)\right) = H$. Thus one of the $\ch_1\left(\cH^i(\ker \psi )\right)$ must be zero, and the other one is equal to $(-1)^iH$.  

First assume $\cH^{-1}(\ker \psi)$ is non-zero, so its rank is less than $2 = \rk(K)$. Note that it cannot be of rank $2$; otherwise, the image of $\psi$ is torsion. Then \eqref{slope-b} implies that $\ch_1\left(\cH^{-1}(\ker \psi )\right)H^2 \leq -\frac{1}{2}H^3$. Thus $\ch_{\leq 1}\left(\cH^{-1} (\ker \psi) \right) = (1, -H)$, and so $\ch_{\leq 1}\left(\cH^{0} (\ker \psi) \right)$ is equal to $(2, 0)$, which is not possible by \eqref{slope-a}. Thus $\cH^{-1}(\ker \psi)$ is zero and $\ker \psi$ is a sheaf.        

We finally claim that $\ker \psi$ is torsion-free. If not, by \eqref{slope-a}, its torsion part is supported in dimension at most $1$. Since $\ker \psi$ is semistable along the wall of phase lower than $1$, it cannot have a subobject of phase $1$. Hence its torsion part is trivial; \textit{i.e.}\ $\ker \psi$ is a torsion-free sheaf. This immediately implies that $\ker \psi = \cI_{\ell}(H)$ for a line $\ell$ on $X$. 
\end{proof}

 The next step is to compute $\cS^2(\cI_{\ell})$. 
\begin{Lem}\label{lem-h0}
		We have $h^0(K_{\ell}(H)) = 3$. 
	\end{Lem}

\begin{proof}
	   We know $h^0(K_{\ell}(H)) = \chi(K_{\ell}(H)) + h^1(K_{\ell}(H)) = 3+ h^1(K_{\ell}(H))$. There is no wall for $K_{\ell}(H)$ crossing the vertical line $\beta = 0$. Therefore, $K_{\ell}(H)$ is semistable along the numerical wall $W(K_{\ell}(H), \cO_X)$. Suppose for a contradiction that there are four linearly independent global sections:   
	   \begin{equation*}
	       s \colon \cO_X^{\oplus 4} \rightarrow K_{\ell}(H). 
	   \end{equation*}
	  Since $\cO_X$ is stable along the wall and has the same phase as $K_{\ell}(H)$, the map $s$ is injective and the cokernel of $s$ is also semistable of the same phase. It is of Chern character $\ch(\text{cok} \ s) = \left(-2, H, -\frac{1}{6}H^2, -\frac{1}{6}H^3\right)$. The numerical wall $W(K_{\ell}(H), \cO_X)$ intersects the horizontal line $\alpha = 0$ at two points with $\beta_1, \beta_2 = 0, -\frac{1}{3}$. Thus by the definition of the heart, we get 
	  \begin{equation}\label{a-1}
	      \mu_H^+\left(\cH^{-1}(\text{cok} \ s)\right) \leq -\frac{1}{3}. 
	  \end{equation}
	  Since $\cH^{0}(\text{cok} \ s)$ is a quotient of $K_{\ell}(H)$, we obtain 
	  \begin{equation*}
	      \frac{1}{2} \leq \mu_H\left(\cH^{0}(\text{cok} \ s)\right). 
	  \end{equation*}
	  Therefore, $\ch_{\leq 1}(\cH^{-1}(\text{cok} \ s)) = (2, -H)$, and $\cH^{0}(\text{cok} \ s)$ is supported in dimension at most $1$. Moreover, \eqref{a-1} implies that $\cH^{-1}(\text{cok} \ s)$ is slope stable; Proposition~\ref{prop-li-ch2} then implies that
          $$\ch_2(\cH^{-1}(\text{cok} \ s))H \leq 0.$$
          Hence
	  \begin{equation*}
	      \ch_2(\text{cok} \ s)H = \ch_2(\cH^{0}(\text{cok} \ s))H -\ch_2(\cH^{-1}(\text{cok} \ s))H \geq 0,
	  \end{equation*}
	which gives  a contradiction. 
\end{proof}

    Therefore, $h^1(K_{\ell}(H)) =0$ and  
	\begin{equation*}
	L_{\cO_X}(K_{\ell}(H)) = \mathrm{cone}(\cO_X^{\oplus 3} \rightarrow K_{\ell}(H)) \eqqcolon J_{\ell}\ , \quad \cS(K_{\ell}) = J_{\ell}[1] \quad \Rightarrow \quad \cS^2(I_{\ell}) = J_{\ell}[3].
	\end{equation*}
	Moreover, 
	\begin{equation*}
	\ch(J_{\ell}) = 	\left(-1, H, -\frac{1}{6}H^2, -\frac{1}{6}H^3\right) = v-w.
	\end{equation*}
	Consider the numerical wall $W(K_{\ell}(H), \cO_X)$. A similar argument as in the proof of Lemma~\ref{lem-h0} implies that 
	\begin{equation*}
	   \mu_H^+\left(\cH^{-1}(J_{\ell})\right) \leq -\frac{1}{3} \quad \text{and} \quad
	      \frac{1}{2} \leq \mu_H\left(\cH^{0}(J_{\ell})\right).   
	\end{equation*}
	Thus $\ch_{\leq 1}\left(\cH^{-1}(J_{\ell}\right)) = (1, -H)$, and $\cH^{0}(J_{\ell})$ is supported in dimension at most $1$. We claim $\cH^{-1}(J_{\ell})$ is a line bundle; otherwise, there is a sheaf $Q$ supported in dimension at most $1$, and we have injections $Q \hookrightarrow \cH^{-1}(J_{\ell})[1] \hookrightarrow J_{\ell}$, which is not possible because of the semistability of $J_{\ell}$ along the numerical wall $W(\cO_X, K_{\ell}(H))$. Thus $\cH^{-1}(J_{\ell}) = \cO_X(-H)$ and $\cH^0(J_{\ell}) = \cO_{\ell}(-H)$. This implies that $J_{\ell}$ is large volume limit stable.    
	
\begin{Def}\label{def-large-stable}
A two-term complex $E \in \Db(X)$ supported in degree 0 and $-1$ is said to be large volume limit stable if $\cH^{-1}(E)$ is a line bundle, $\cH^0(E)$ is a sheaf supported in dimension at most $1$, and $\Hom(\Coh_{\leq 1}(X), E) = 0$. 
\end{Def}	

By~\cite[Lemma 3.12 and Lemma 3.13(ii)]{toda-bogomolov}, a complex $E \in \Db(X)$ is large volume limit stable if and only if $E$ lies in $\Coh^{\beta}(X)$ and is $\sigma_{\alpha, \beta}$-stable for $\beta > \mu_H(E)$ and $\alpha \gg 0$.

\begin{proof}[Proof of Theorem~\ref{thm-moduli-spaces}]
Any slope-semistable sheaf $E$ of class $v$ is isomorphic to $\cI_{\ell}$ for a line $\ell$ in $X$. By Lemma~\ref{lem-s-v}, $\cI_{\ell}$ is $\sigma$-stable for any $\cS$-invariant stability condition $\sigma$ on $\ku(X)$. 

Conversely, let $E \in \ku(X)$ be a $\sigma$-stable object of class $[\cI_{\ell}]$. By Corollary~\ref{cor-s-invariant-ku}, we may assume $\sigma = \sigma(\alpha, \beta)$, where $(\alpha, \beta) \in V \cap \Gamma$ with $\Gamma$ the hyperbola with equation $\alpha^2 = \beta^2 -\frac{2}{3}$. One can easily compute that 
\begin{equation} \label{eq_Zofline=-1}
    Z(\alpha, \beta)(E) =-1. 
\end{equation}
So $E$ is of maximum phase; thus $E$ is $\sigma^0_{\alpha, \beta}$-semistable. Since $E$ lies in $\ku(X)$ (so we have  $\chi(\cO_X(kH), E) = 0$ for $k=0, 1$) and is of class $[\cI_{\ell}]$, one can easily show that $E$ is of Chern character $\pm v$. Proposition~\ref{prop-rotated-stability} implies that either
\begin{inparaenum}\item\label{inpar1} $E$ is $\sigma_{\alpha, \beta}$-semistable, or \item\label{inpar2} $E$ lies in an exact triangle 
\begin{equation*}
    F[1] \rightarrow E \rightarrow T,
\end{equation*}
where $F \in \mathcal{F}_{\alpha, \beta}$ and $T \in \Coh_0(X)$.
\end{inparaenum} In the latter case, $F[1]$ is of phase $1$ with respect to $\sigma^0_{\alpha, \beta}$, so it is $\sigma_{\alpha, \beta}$-semistable by Proposition~\ref{prop-rotated-stability}. We have
\begin{equation*}
    \ch(F) = -\ch(E) +\ch(T) = \left(1, 0, -\frac{1}{3}H^2 ,\ch_3(T) \right).
\end{equation*}
Let $(\alpha_0, \beta_0)$ be the intersection point of $\partial V$ with $\Gamma$. The closedness of semistability implies that $E$ in case~\eqref{inpar1} (or $F$ in case~\eqref{inpar2}) is $\sigma_{\alpha_0, \beta_0}$-semistable. We know the point $(\alpha_0, \beta_0)$ lies on the numerical wall $W(E, \cO_X(-H)[1])$. In case~\eqref{inpar2}, applying Theorem~\ref{thm-li-ch3} to $F$ at the limiting point $(\alpha , \beta) \rightarrow (0, -1)$, we obtain that 
\begin{equation*}
\ch_3(T) = \ch_3(F) \leq \frac{5}{9}, 
\end{equation*}
which gives $T = 0$, so $E =F[1]$ and we are reduced to case~\eqref{inpar1}. In case~\eqref{inpar1},  Lemma~\ref{lem-numerial} below implies that $E$ is $\sigma_{\alpha, \beta}$-stable for $\alpha \gg 0$, so $E$ is a slope-stable sheaf up to a shift, by~\cite[Lemma 2.7]{bayer:the-space-of-stability-conditions-on-abelian-threefolds}. This completes the proof of part~\eqref{thm-mod-sp-a}. 

\bigskip

For part~\eqref{thm-mod-sp-b}, take a slope-stable sheaf $E$ of class $w$. By Proposition~\ref{prop_slopestableclassw}, there is a line $\ell$ on~$X$ such that $E = K_{\ell}$. Thus $E$ is in $\ku(X)$ and is $\sigma$-stable for any $\cS$-invariant stability condition~$\sigma$. Conversely, let $E$ be a $\sigma$-stable object in $\ku(X)$ of class $[\cS(\cI_{\ell})]$. Then $\cS^{-1}(E)$ is $\sigma$-stable of class $[\cI_{\ell}]$. Hence by part~\eqref{thm-mod-sp-a}, $\cS^{-1}(E)$ is, up to a shift, isomorphic to the ideal sheaf of a line $\ell$ on $X$. Thus $E$ is isomorphic to $\cS(\cI_{\ell})$ and so is a slope-stable sheaf up to a shift.     

\bigskip 

Finally, we prove part~\eqref{thm-mod-sp-c}. Take a large volume limit stable complex $J$ of class $v-w$. We know there is no wall for $J$ above the numerical wall $W(\cO_X, J)$. Thus $J$ is $\sigma_{\alpha, \beta}$-stable of phase lower than $1$ if $-1<\beta < -\frac{1}{3}$. Sheaves supported in dimension zero are $\sigma_{\alpha, \beta}$-semistable of phase $1$, so $\Hom(\Coh_0(X), J) = 0$. One can easily check that $J \in \mathcal{T}_{\alpha, \beta}$ when $\beta \rightarrow -1$; thus Proposition~\ref{prop-rotated-stability} implies $J$ is $\sigma^0_{\alpha, \beta}$-stable. We claim that $J$ lies in the Kuznetsov component $\ku(X)$. By definition, $J$ lies in the exact triangle 
\begin{equation*}
    \cO_X(-H)[1] \rightarrow J \rightarrow \cO_{\ell}(-H).
\end{equation*}
We know $\Hom^i(\cO_X(H) , \cO_X(-H))$ vanishes for $i \neq 3$ and is isomorphic to $\mathbb{C}$ for $i=3$. Also, $\Hom^i(\cO_X(H) , \cO_{\ell}(-H)) = 0$ vanishes for $i \neq 1$ and is isomorphic to $\mathbb{C}$ for $i=1$. Thus $\hom^i(\cO_X(H), J) = 0$ for every $i \in \mathbb{Z}$. We also have 
$$
\Hom^i(\cO_X, \cO_X(-H)) = \Hom^i(\cO_X , \cO_{\ell}(-H)) = \Hom^i(\cO_X, J) = 0  $$
for all $i \in \mathbb{Z}$. Therefore, $J \in \ku(X)$, and so $J$ is $\sigma(\alpha, \beta)$-stable when $-1<\beta < -\frac{1}{3}$. 

Conversely, take a $\sigma$-stable object $E \in \ku(X)$ of class $[\cS^2(\cI_{\ell})]$ for an $\cS$-invariant stability condition $\sigma$. Thus $\cS^{-1}(E)$ is $\sigma$-stable of class $[K_{\ell}]$, so by the second part, there is a line $\ell$ in $X$ such that $\cS^{-1}(E) = K_{\ell}$. Thus $E$ is isomorphic to $J_{\ell}$ up to  shift. This completes the proof.   
\end{proof}

\begin{Lem}\label{lem-numerial}
Let $E$ be a $\sigma_{\alpha, \beta}$-semistable object of Chern character $v$ for $(\alpha, \beta)$ along the numerical wall $W(E, \cO_X(-H)[1])$. Then $E$ is $\sigma_{\alpha, \beta}$-stable for $\alpha \gg 0$ and $\beta > 0$.  
\end{Lem}

\begin{proof}
Suppose for a contradiction that $W(E, \cO_X(-H)[1])$ is an actual wall and $F_1 \rightarrow E \rightarrow F_2$ is a destabilising sequence. Let $\ch_{\leq 2}(F_1) = (r, cH, sH^2)$. Since the boundary of the wall intersects the vertical line $\beta =-1$, we must have $\Im[Z_{\alpha = 0, \beta =-1}(F_i)] \geq 0$; thus 
\begin{equation*}
    0 \leq c+r \quad \text{and} \quad0 \leq -c+1-r\, . 
\end{equation*}
Hence by relabelling $F_1$ and $F_2$, we may assume $c= -r$. Since the $F_i$ have the same phase as $E$ along the wall, one can easily show that
$s = \frac{1}{2}r$. Therefore, $\ch_{\leq 2}(F_2) = \left(1-r,\ rH,\ H^2\left(-\frac{1}{3}-\frac{1}{2}r\right)\right)$ and 
\begin{equation*}
    \Delta_H(F_2) = (H^3)^2\left(\frac{r}{3} + \frac{2}{3}\right). 
\end{equation*}
By~\cite[Corollary 3.10]{bayer:the-space-of-stability-conditions-on-abelian-threefolds}, we must have $0 \leq \Delta_H(F_2) < \Delta_H(E) = \frac{2}{3}(H^3)^2$. Since $\Delta_H \in \frac{1}{3}(H^3)^2 \mathbb{Z}$, we obtain $r =-1$ or $r=-2$. In both cases, $F_2$ is of positive rank and $F_1$ is of negative rank, so $F_1$ must be the subobject in the destabilising sequence to have phase bigger than $E$ above the wall.   

Since $\Delta_H(F_1) = 0$, $F_1$ is $\sigma_{\alpha, \beta}$-semistable for any $\alpha > 0$ and $\beta > -1$; see~\cite[Corollary 3.11(a)]{bayer:the-space-of-stability-conditions-on-abelian-threefolds}. Thus $\cH^{-1}(F_1)$ is a slope-stable reflexive sheaf, and $\cH^0(F_1)$ is supported in dimension at most $1$. Since $\Delta_H(F_1) = 0$, there is no room for $\ch_2(\cH^0(F_1))$, so $\cH^0(F_1)$ is indeed supported in dimension zero. Applying Theorem~\ref{thm-li-ch3} to $F_1(H)$ shows that for $\beta >0$, we have
\begin{equation}\label{f1}
0 \leq 3r\beta \left(\ch_3\left(\cH^0\left(F_1(H)\right)\right) - \ch_3\left(\cH^{-1}\left(F_1(H)\right)\right)\right).   
\end{equation}
Moreover, since $\cH^{-1}(F_1(H))$ is a $\mu_H$-stable sheaf, Lemma~\ref{lem-large-volume} implies that it is $\sigma_{\alpha', \beta'}$-stable for $\beta' <0$ and $\alpha' \gg 0$. Applying Theorem~\ref{thm-li-ch3} to $\cH^{-1}(F_1)$ at this region shows that 
\begin{equation*}
    0 \leq -3r\beta' \ch_3\left(\cH^{-1}\left(F_1(H)\right)\right) \quad \Rightarrow \quad 0 \leq -\ch_3\left(\cH^{-1}\left(F_1(H)\right)\right).
\end{equation*}
Combining this with \eqref{f1} shows that $\ch_3\left(\cH^{i}\left(F_1(H)\right)\right) = 0$ for $i =-1, 0$. Therefore, \smash{$\cH^0(F_1)=0$} and $[F_1] = r[\cO_X(-H)]$; thus $F_1 \cong \cO_X(-H)^{\oplus {-r}}[1]$. 

If $r = -1$ and $F_2$ is $\sigma_{\alpha, \beta}$-stable along the wall, then $F_2$ will be $\sigma_{\alpha, \beta}$-stable for $\alpha \gg 0$ and $\beta > -\frac{1}{2}$ because there is no wall for it crossing the vertical line $\beta = -1$. Hence $F_2$ is a slope-stable sheaf of character $\ch_{\leq 2}(F_2) = \left(2, -H, \frac{1}{6}H^2\right)$, which is not possible by Proposition~\ref{prop-li-ch2}. Thus $F_2$ is strictly semistable along the wall, and the same argument as above implies that one of its destabilising subobject is $\cO_X(-H)[1]$. Hence we may assume $r= -2$. Then $\Delta_H(F_2) = 0$; thus by~\cite[Corollary 3.11(a)]{bayer:the-space-of-stability-conditions-on-abelian-threefolds}, $F_2$ is a slope-stable sheaf, which again is not possible by Proposition~\ref{prop-li-ch2}. 
\end{proof}

\section{Stability conditions on the Kuznetsov component via conic fibrations} \label{sec_kuzconicfibr}

Let $X$ be a cubic threefold as before. Fix a line $\ell_0$ in $X$, and consider the linear projection to a disjoint $\PP^2$ in the $\PP^4$ containing $X$, which induces a conic fibration structure on the blowup of $X$ along $\ell_0$. Let $\cB_0$ (resp.\ $\cB_1$) be the sheaf of even (resp.\ odd) parts of the Clifford algebra corresponding to the conic fibration as in~\cite[Section 3]{kuznetsov-quadric-fibration}. We denote by $\Coh(\PP^2, \cB_0)$ the abelian category of right coherent $\cB_0$-modules and by $\Db(\PP^2, \cB_0)$ its bounded derived category. By~\cite[Section 2.1]{macri:categorical-invarinat-cubic-threefolds}, there is a semiorthogonal decomposition
\begin{equation}\label{semi}
    \Db(\PP^2, \cB_0)=\langle \Xi(\ku(X)), \cB_1 \rangle,
\end{equation}
where $\Xi: \ku(X) \to \Db(\PP^2, \cB_0)$ is a fully faithful embedding. We set 
$$\ku(\PP^2,\cB_0):=\Xi(\ku(X)),$$
which is equivalent to $\ku(X)$ by definition. In~\cite{macri:categorical-invarinat-cubic-threefolds,macri:acm-bundles-cubic-3fold}, the authors use the se\-mi\-or\-thog\-o\-nal decomposition \eqref{semi} to construct stability conditions on $\ku(X)$. In this section, we first summarise their construction and then show that all these stability conditions on $\ku(X)$ are $\cS$-invariant and lie in the same orbit with respect to $\widetilde{\text{GL}}^+_2$-action.  

\subsection{Stability conditions on non-commutative  $\PP^2$} The Chern character of an object $E \in \Db(\PP^2,\cB_0)$ is
$$
\ch(\text{Forg}(E)) \in K(\PP^2) \otimes \Q \cong \Q^{\oplus 3},
$$
where $\text{Forg} \colon \Db(\PP^2,\cB_0) \to \Db(\PP^2)$ is the forgetful functor; by abuse of notation, we denote it by $\ch(E)$ for simplicity.
Recall that by~\cite[Proposition 2.12]{macri:categorical-invarinat-cubic-threefolds}, we have
\begin{equation}
    \label{eq_numgrgr}
\mathcal{N}(\PP^2, \cB_0) = \Z[\cB_{-1}] \oplus \Z[\cB_0] \oplus \Z[\cB_1],    
\end{equation}
and the Chern characters 
\begin{equation*}
\ch(\cB_{-1}) = \left(4, -7, \frac{15}{2}\right),  \quad \ch(\cB_0) = \left(4, -5, \frac{9}{2}\right), \quad \ch(\cB_1) = \left(4, -3, \frac{5}{2}\right)
\end{equation*}
are linearly independent. For objects in $\Db(\PP^2, \cB_0)$, we have the Euler characteristic 
\begin{equation*}
\chi(-, -) = \sum_{i} (-1)^i \hom^i_{\Db(\PP^2, \cB_0)}(-, -).  
\end{equation*}
Given $[F] = x[\cB_{-1}] +y[\cB_{-1}] +z [\cB_{-1}] \in \mathcal{N}(\PP^2, \cB_0)$ with $\ch(F) = (r, c_1, \ch_2)$, ~\cite[Remark 2.2(iv)]{macri:acm-bundles-cubic-3fold} gives 
\begin{align*}
\chi(F, F) = \ &\  x^2+ y^2+z^2+3xy + 3yz + 6xz\\
= \ & \ -\frac{7}{64}r^2 - \frac{1}{4}c_1^2 +\frac{1}{2}r\ch_2. 
\end{align*}
Slope stability for torsion-free sheaves in $\Coh(\PP^2, \cB_0)$ is defined via $$\rk(F) \coloneqq \rk(\text{Forg}(F)), \quad \deg(F) = c_1(\text{Forg}(F)).c_1(\cO_{\PP^2}(1)).$$

\begin{Lem}[\protect{\cite[Theorem 8.3]{bayer:stability-conditions-kuznetsov-component}}]
Any slope-semistable torsion-free sheaf $E$ in $\Coh(\PP^2, \cB_0)$ of character $(r, c_1, \ch_2)$ satisfies the quadratic inequality  
\begin{equation}\label{quadratic}
    Q(E) \coloneqq c_1^2 -2r\ch_2 + \frac{11}{16}r^2 \geq 0. 
\end{equation}
\end{Lem}

\begin{proof}
If $E$ is slope stable, then~\cite[Lemma 2.4]{macri:acm-bundles-cubic-3fold} implies that $\chi(E, E) \leq 1$. Since $r = \rk(E)$ is a multiple of $4$ by~\cite[Lemma 2.13(i)]{macri:categorical-invarinat-cubic-threefolds}, we obtain  
\begin{equation*}
 \chi(E, E) =  -\frac{7}{64}r^2 - \frac{1}{4}c_1^2 +\frac{1}{2}r\ch_2 \leq \frac{r^2}{16}, 
\end{equation*}
which implies the claim. If $E$ is strictly slope semistable with stable factors $\{E_i\}_{i \in I}$, we get 
\begin{equation*}
\frac{\ch_2(E_i)}{r(E_i)} \leq \frac{1}{2}\left(\frac{c_1(E_i)}{r(E_i)}\right)^2 + \frac{11}{32}\ =\ \frac{1}{2}\left(\frac{c_1}{r}\right)^2 + \frac{11}{32}, 
\end{equation*}
which shows \eqref{quadratic} holds for $E$. 
\end{proof}

Consider the open subset 
\begin{equation*}
    U \coloneqq \left\{(b,w) \in \mathbb{R}^2 \colon \ w > \frac{b^2}{2} + \frac{11}{32}  \right\}.
\end{equation*}
By \eqref{quadratic} we have the following result, which is well known in the commutative setting by~\cite{bayer:bridgeland-stability-conditions-on-threefolds, bayer:the-space-of-stability-conditions-on-abelian-threefolds} (see also~\cite[Lemma 4.4]{PPZ}).

\begin{Prop}
There is a continuous family of Bridgeland stability conditions on $\Db(\PP^2,\cB_0)$ parametrised by $U$ given by  
\begin{equation*}
(b,w) \in U \mapsto \overline{\sigma}_{b,w} =\big( \Coh^b(\PP^2, \cB_0) ,\ \overline{Z}_{b,w} = -\ch_2 + w \ch_0 +i (\ch_1 -b\ch_0)\big).\footnote{In~\cite{macri:acm-bundles-cubic-3fold}, authors consider only the ray $\overline{\sigma}_{b=-1 , w}$ for $w > \frac{27}{32}$, where $w = \frac{23}{32} +2m^2$ in their notation.}
\end{equation*}
\end{Prop}

\subsection{Induced stability conditions} \label{sec-inducedstabcond}
By~\cite[Proposition 2.9(ii)]{macri:categorical-invarinat-cubic-threefolds}, the Serre functor $S_{\Db(\PP^2, \cB_0)}$ of $\Db(\PP^2, \cB_0)$ is equal to $- \otimes_{\cB_0} \cB_{-1}[2]$. Remark 9.4 of \cite{bayer:stability-conditions-kuznetsov-component} shows that $\cB_0$ and $\cB_1$ are slope-stable $\cB_0$-modules with $\mu(\cB_0)=-\frac{5}{4}$ and $\mu(\cB_1)=-\frac{3}{4}$, respectively. Then $\cB_1$ and $\cB_0[1]=S_{\Db(\PP^2, \cB_0)}(\cB_1)[-1]$ are in $\Coh^b(\PP^2, \cB_0)$ for $-\frac{5}{4}\leq b < -\frac{3}{4}$. Since $\overline{\sigma}_{b, w}$ is a stability condition on $\Db(\PP^2, \cB_0)$, we see that the conditions in~\cite[Proposition 5.1]{bayer:stability-conditions-kuznetsov-component} are all satisfied, as we wanted. 

\begin{Prop}\label{prop.2-dim}
For any $(b,w) \in U$ with $-\frac{5}{4}\leq b < -\frac{3}{4}$, the pair
\begin{equation}\label{bar-1}
\overline{\sigma}(b,w) \coloneqq \big(\cA_b \coloneqq \Coh^b(\PP^2, \cB_0) \cap \ku(\PP^2, \cB_0) ,\ \overline{Z}_{b,w}|_{\mathcal{N}(\cA_b)} \big)
\end{equation}
is a stability condition on $\ku(\PP^2, \cB_0)$. With respect to $\widetilde{\emph{GL}}^+_2$-action, they all lie in the same orbit as the stability condition\footnote{The stability condition $\overline{\sigma}$ is first constructed in~\cite[Section 3]{macri:categorical-invarinat-cubic-threefolds}.}   
\begin{equation}\label{bar}
    \overline{\sigma} = \left(\cA_{-\frac{5}{4}}\,,\, \overline{Z} = \ch_0 + i\left(\ch_1 + \frac{5}{4}\ch_0 \right) \right). 
\end{equation}
\end{Prop}

\begin{proof}
The first claim follows from~\cite[Proposition 5.1]{bayer:stability-conditions-kuznetsov-component} (see Proposition~\ref{prop_inducstab}). Arguing as in~\cite[Proposition and Definition 2.15]{LPZ:elliptic}, we have the relation
$$\ch_2(F)=-\ch_1(F) -\frac{3}{8}\rk(F)$$
for every $F \in \ku(\PP^2,\cB_0)$. Thus we can rewrite
$$\Re [\overline{Z}_{b, w}|_{\mathcal{N}(\cA_b)}]=\ch_1 + \left( \frac{3}{8}+w \right) \ch_0.$$
As a consequence,
\begin{align*}
\overline{\mu}_{b, w}|_{\mathcal{N}(\cA_b)} = -\frac{\Re [\overline{Z}_{b, w}|_{\mathcal{N}(\cA_b)}] }{\Im [\overline{Z}_{b, w}|_{\mathcal{N}(\cA_b)}]  }&=\frac{-\ch_1  -\left( \frac{3}{8}+w \right) \ch_0}{\ch_1 -b\ch_0}\\
&= \frac{ -\ch_1  -\left( \frac{3}{8}+w \right) \ch_0 \, + \, \ch_1 -b\ch_0 }{\ch_1 -b\ch_0 } -1 \\
& = - \frac{\left(\frac{3}{8} +w +b \right)\ch_0}{\ch_1 -b\ch_0} -1.
\end{align*}
Since $\frac{3}{8} +w +b >0$ for $(b,w) \in U$, the stability function $\overline{Z}_{b, w}|_{\mathcal{N}(\cA_b)}$ induces the same ordering on $\cA_b$ as the function 
\begin{equation*}
    Z_b \coloneqq \ch_0 + i (\ch_1 -b \ch_0). 
\end{equation*} 
Thus the stability condition $\left(\cA_b, \overline{Z}_{b, w}|_{\mathcal{N}(\cA_b)} \right)$ lies in the same orbit as $\sigma_b = \left(\cA_b, Z_b\right)$ with respect to $\widetilde{\text{GL}}^+_2$-action. We know $Z_b = T_b^{-1} \circ Z_{-\frac{5}{4}}$, where 
$$
T_b = \begin{pmatrix}
1 & 0\\
b + \frac{5}{4} & 1
\end{pmatrix}. 
$$
Since $b+ \frac{5}{4} \geq 0$, there exists a cover $g\_b \coloneqq (T_b, f_b) \in \widetilde{\text{GL}}^+_2(\R)$ such that 
$$
\mathcal{P}_{\sigma_{-\frac{5}{4}}}\left(f_b(0, 1]\right) \subset \langle \cA_{-\frac{5}{4}}, \cA_{-\frac{5}{4}}[1] \rangle.
$$
Also, an easy computation shows that $\Coh^b(\PP^2, \cB_0) \subset \langle \Coh^{-\frac{5}{4}}(\PP^2, \cB_0), \Coh^{-\frac{5}{4}}(\PP^2, \cB_0)[1] \rangle$ (see for instance~\cite[Lemma 3.7]{pertusi:some-remarks-fano-threefolds-index-two}). Restricting to the heart $\cA_b$, this implies
$$\cA_b \subset \langle \cA_{-\frac{5}{4}}, \cA_{-\frac{5}{4}}[1] \rangle$$
by~\cite[Lemma 4.3]{bayer:stability-conditions-kuznetsov-component}. Thus the stability conditions $\sigma_{b}$ and $\sigma_{-\frac{5}{4}}\,.\, g\_b$ have the same central charge and satisfy the conditions of~\cite[Lemma 8.11]{bayer:the-space-of-stability-conditions-on-abelian-threefolds}; hence they are the same. 
\end{proof}

\subsection{Serre invariance} \label{sec-Srreinvariance}
In this section we show the following result. 

\begin{Thm}\label{thm-S-B0}
The stability condition $\overline{\sigma} = \left(\cA_{-\frac{5}{4}} , \overline{Z} \right)$ from \eqref{bar} on $\ku(\PP^2, \cB_0)$ is Serre invariant.
\end{Thm}

Therefore, Proposition~\ref{prop.2-dim} implies that all stability conditions $\overline{\sigma}(b, w)$ from \eqref{bar-1} are Serre invariant. By~\cite{Bondal} (see also~\cite[Proposition 3.8]{kuznetsov-derived-category-cubic-3folds-V14}),  the Serre functor $\cS_{\cB_0}$ of $\ku(\PP^2,\cB_0)$ satisfies
$$\cS^{-1}_{\cB_0} \cong L_{\cB_1}(- \otimes \cB_1)[-2] \cong (- \otimes \cB_1) \circ L_{\cB_0}(-)[-2].$$
We first describe the image of the heart after the action of $\cS_{\cB_0}^{-1}$.

\begin{Prop}\label{prop-heart}
We have $L_{\cB_0}(\cA_{-\frac{5}{4}}) \otimes \cB_{1} \subset \langle \cA_{-\frac{5}{4}} , \cA_{-\frac{5}{4}}[1] \rangle$. 
\end{Prop}

\begin{proof}
Let $E \in \cA_{-\frac{5}{4}} = \Coh^{-\frac{5}{4}}(\PP^2, \cB_0) \cap \ku(\PP^2, \cB_0)$, and set $E':=L_{\cB_0}(E)$. We first show that 
\begin{equation}\label{claim}
    E' \otimes \cB_1=L_{\cB_0}(E) \otimes \cB_1 \in \langle \Coh^{-\frac{5}{4}}(\PP^2, \cB_0), \Coh^{-\frac{5}{4}}(\PP^2, \cB_0)[1] \rangle. 
\end{equation}
By Serre duality and the fact that $\cB_0[1], \cB_{-1}[1] \in \Coh^{-\frac{5}{4}}(\PP^2, \cB_0)$, we have that $\Hom^i(\cB_0[1], E)$ can be non-zero only for $i=0,1,2$. Thus $E'$ is defined by the triangle
$$\cB_0[1]^{\oplus k_1} \oplus \cB_0^{\oplus k_2} \oplus \cB_0[-1]^{\oplus k_3} \xrightarrow{\ev}  E \to E'.$$
Since $\cB_0[1]$ is stable of phase $1$ with respect to $\overline{\sigma}_{b =-\frac{5}{4}, w}$,
the map $f$ in the following commutative diagram is injective in $\Coh^{-\frac{5}{4}}(\PP^2, \cB_0)$:
$$
\xymatrix{
\cB_0[1]^{\oplus k_1} \ar[d] \ar[r]^{f} & E \ar[d]^{\text{id}} \ar[r] & G \ar[d]\\
\cB_0[1]^{\oplus k_1} \oplus \cB_0^{\oplus k_2} \oplus \cB_0[-1]^{k_3} \ar[d] \ar[r] & E \ar[d] \ar[r] & E' \ar[d] \\
\cB_0[1]^{\oplus k_2} \oplus \cB_0[-1]^{k_3} \ar[r] & 0 \ar[r] & \cB_0[1]^{\oplus k_2} \oplus \cB_0^{k_3}.
}
$$
Hence, we obtain $G \in \Coh^{-\frac{5}{4}}(\PP^2, \cB_0)$. Tensoring the third column by $\cB_1$ gives the exact triangle 
$$G \otimes \cB_1 \to E' \otimes \cB_1 \to \cB_1^{\oplus k_2}[1] \oplus \cB_1^{\oplus k_3}.$$
By  Lemma~\ref{lem-tensor} below, we have $G \otimes \cB_1 \in \langle \Coh^{-\frac{5}{4}}(\PP^2, \cB_0), \Coh^{-\frac{5}{4}}(\PP^2, \cB_0)[1] \rangle$; thus  claim \eqref{claim} follows. 

Since $L_{\cB_0}(E) \otimes \cB_1 = L_{\cB_1}(E \otimes \cB_1)$, we know $L_{\cB_0}(E) \otimes \cB_1 \in \ku(\PP^2, \cB_0)$. If we consider the HN filtration of $L_{\cB_0}(E) \otimes \cB_1$ with respect to $\overline{\sigma}_{b=-\frac{5}{4}, w}$, claim~\eqref{claim} implies that the maximum and minimum phases satisfy 
\begin{equation*}
    0< \phi^{-}_{\overline{\sigma}_{b=-\frac{5}{4}, w}}(L_{\cB_0}(E) \otimes \cB_1) \leq \phi^{+}_{\overline{\sigma}_{b=-\frac{5}{4}, w}}(L_{\cB_0}(E) \otimes \cB_1) \leq 2. 
\end{equation*}
Then the definition of stability condition $\overline{\sigma}$ implies that
\begin{equation*}
\begin{split}
    0< \phi^{-}_{\overline{\sigma}_{b=-\frac{5}{4}, w}}(L_{\cB_0}(E) \otimes \cB_1) \leq \phi^{-}_{\overline{\sigma}}(L_{\cB_0}(E) \otimes \cB_1)& \leq \phi^{+}_{\overline{\sigma}}(L_{\cB_0}(E) \otimes \cB_1)\\
    &\leq 
    \phi^{+}_{\overline{\sigma}_{b=-\frac{5}{4}, w}}(L_{\cB_0}(E) \otimes \cB_1) \leq 2. 
\end{split}    
\end{equation*}
Therefore, 
\begin{equation*}
    L_{\cB_0}(E) \otimes \cB_1 \in \langle \Coh^{-\frac{5}{4}}(\PP^2, \cB_0)\cap \ku(\PP^2, \cB_0), \Coh^{-\frac{5}{4}}(\PP^2, \cB_0)\cap \ku(\PP^2, \cB_0)[1] \rangle,
\end{equation*} 
as claimed. 
\end{proof}   

\begin{Lem}\label{lem-tensor}
We have $\Coh^b(\PP^2,\cB_0) \otimes \cB_1 \subset \langle \Coh^b(\PP^2,\cB_0), \Coh^b(\PP^2,\cB_0)[1] \rangle$.
\end{Lem}

\begin{proof}
Since $\cB_1$ is a flat $\cB_0$-module, by \eqref{eq_numgrgr}, for $F \in \Coh(\PP^2, \cB_0)$, we have that
$$\ch_1(F \otimes \cB_1)=\ch_1(F)+\frac{1}{2}\rk(F).$$
If $F \in \Coh(\PP^2, \cB_0)$ is slope semistable with $\mu(F)> b$, then $F \otimes \cB_1$ is slope semistable with slope $\mu(F)+\frac{1}{2}> b$. Then $F \otimes \cB_1$ is in $\Coh^b(\PP^2,\cB_0)$. Otherwise, assume $F \in \Coh(\PP^2, \cB_0)$ is slope semistable with $\mu(F) \leq b$. If $\mu(F)+\frac{1}{2} \leq b$, then $F \otimes \cB_1[1] \in \Coh^b(\PP^2,\cB_0)$, while if $\mu(F)+\frac{1}{2} > b$, then $F \otimes \cB_1 \in \Coh^b(\PP^2,\cB_0)$ and thus $F \otimes \cB_1[1] \in \Coh^b(\PP^2,\cB_0)[1]$. This implies the claim.
\end{proof}

\begin{proof}[Proof of Theorem~\ref{thm-S-B0}]
Since $\cS_{\cB_0}^2[-3] = \cS_{\cB_0}^{-1}[2]$, we only need to find a $g \coloneqq (T, f) \in \widetilde{\text{GL}}^+_2(\R)$ such that
\begin{equation*}
 \cS_{\cB_0}^{-1}[2]\, .\, \overline{\sigma}   = \overline{\sigma}\,.\,g.
\end{equation*}
We know $\cS_{\cB_0}^{-1}[2]\, .\, \overline{\sigma}  = \left( \overline{Z}  \circ \cS_{\cB_0}[-2]\, ,\,\cS_{\cB_0}^{-1}[2]\left(\cA_{-\frac{5}{4}} \cap \ku(\PP^2, \cB_0)\right)\right)$. By~\cite[Remark 2.2 and Proposition 2.12]{macri:acm-bundles-cubic-3fold}, the classes $v_1 \coloneqq [\Xi(\cI_\ell)]=[\cB_1]-[\cB_0]$ and $v_2 \coloneqq [\Xi(S_{\ku(X)}(\cI_\ell))]=2[\cB_0]-[\cB_{-1}]$ are the basis of $\cN(\ku(\PP^2,\cB_0))$. Also, \cite[Lemma 1.30]{huyb-book-FM} implies $\cS_{\cB_0}(\Xi(\cI_\ell)) = \Xi(\cS_{\ku(X)}(\cI_{\ell})) $ and 
$$
\left[\cS_{\cB_0}\left(\Xi\left(\cS_{\ku(X)}(\cI_\ell)\right)\right)\right] = \left[\Xi\left(\cS^2_{\ku(X)}(\cI_{\ell}) \right)\right] = \left[\Xi\left(\cS_{\ku(X)}(\cI_{\ell}) \right)\right]- \left[\Xi(\cI_{\ell})\right]. 
$$
Therefore, $\overline{Z}  \circ \cS_{\cB_0} = T^{-1} \circ  \overline{Z} $ for a linear invertible function $T \colon \mathbb{R}^2 \rightarrow \mathbb{R}^2$ defined as 
\begin{equation*}
T^{-1}\left[\overline{Z}(v_1)\right] =  \overline{Z} (v_2) \quad \text{and} \quad T^{-1}\left[ \overline{Z} (v_2)\right] =  \overline{Z} (v_2) - \overline{Z} (v_1). 
\end{equation*}
Thus
\begin{equation*}
    T\left[ \overline{Z}(v_1)\right] = \overline{Z}(v_1) -\overline{Z}(v_2) \quad \text{and} \quad T \left[\overline{Z}(v_2)\right] =  \overline{Z} (v_1). 
\end{equation*}
We know $\overline{Z}(v_1) = 2 i$ and $\overline{Z}(v_2) = 4+ 2i$, so one can easily check that the linear function $T$ with respect to the standard basis can be represented as  
$$
\begin{pmatrix}
1 & -2\\
\frac{1}{2} & \hphantom{{-}}0
\end{pmatrix}. 
$$
Thus there exists a cover $(T, f) \in \widetilde{\text{GL}}^+_2(\R)$ such that $f(0, 1) \subset (0, 2)$, so 
$$\mathcal{P}_{\overline{\sigma}}(f(0, 1]) \subset \langle \cA_{-\frac{5}{4}}, \cA_{-\frac{5}{4}}[1] \rangle.$$ 
Thus the stability conditions $\sigma \cdot g$ and $((- \otimes \cB_1) \circ L_{\cB_0}) \cdot \sigma$ have the same stability functions, and the condition of~\cite[Lemma 8.11]{bayer:the-space-of-stability-conditions-on-abelian-threefolds} holds for their hearts by Proposition~\ref{prop-heart}. Hence they are the same stability conditions. 
\end{proof}

\begin{Cor}
\label{cor_barsigmavssigmaalphabeta}
The stability condition $\overline{\sigma}$ from \eqref{bar} and the stability conditions $\overline{\sigma}(b, w)$ defined in \eqref{bar-1} are in the same orbit as $\sigma(\alpha,\beta)$ defined in Theorem~\ref{thm_stabcondinduced} with respect to the  $\widetilde{\emph{GL}}^+_2(\R)$-action.    
\end{Cor} 

\begin{proof}
As noted in Corollary~\ref{cor-s-invariant-ku}, the conditions of Theorem~\ref{thm_uniqueSinv} are satisfied by $\ku(X)$. Then the statement follows from Theorem~\ref{thm-S-B0}, Theorem~\ref{thm_stabcondinduced} and the uniqueness implied by Theorem~\ref{thm_uniqueSinv}. 
\end{proof}   

\begin{Rem}
Note that all the known stability conditions on $\ku(X)$ are Serre invariant. An interesting question would be to understand whether the property of $S$-invariance holds for every stability condition on $\ku(X)$. A positive answer would imply that there is a unique $\widetilde{\text{GL}}^+_2(\R)$-orbit of stability conditions by Theorem~\ref{thm_uniqueSinv}, giving a complete description of the stability manifold of $\ku(X)$, in analogy to the case of the bounded derived category of a genus at least $2$ curve studied in~\cite{macri:stability-conditions-on-curves}.
\end{Rem}

\section{Ulrich bundles}
\label{sec_Ulrich}

Let $X$ be a smooth cubic threefold as before. We denote by $\mathfrak{M}^{(s)U}_{d}$ the moduli of (stable) Ulrich bundles. By~\cite[Theorem B]{macri:acm-bundles-cubic-3fold}, the moduli space $\mathfrak{M}^{sU}_{d}$ of stable Ulrich bundles of rank $d$ on $X$ is non-empty and smooth of dimension $d^2+1$. In this section we apply our result on Serre-invariant stability conditions to show the following.

\begin{Thm}
\label{thm_ulrich}
The moduli space $\mathfrak{M}^{U}_{d}$ of Ulrich bundles of rank $d$ is irreducible.
\end{Thm}

Recall that an Ulrich bundle $E$ of rank $d \geq 2$ on $X$ is a vector bundle on $X$ satisfying $H^i(X, E(jH))=0$ for all $i=1,2$, $j \in \Z$ and such that the graded module $\oplus_{m \in \Z} H^0(X, E(mH))$ has $3d$ generators in degree $1$. Proposition 2.9 of \cite{ulrich} shows that any Ulrich bundle $E$ is Gieseker semistable, and if it is Gieseker stable, then it is $\mu$-stable. Moreover, if an Ulrich bundle $E$ is Gieseker strictly semistable, then its stable factors are also Ulrich bundles. Thus~\cite[Lemma 2.19]{macri:acm-bundles-cubic-3fold} implies that 
$$\ch(E)=\left(d, 0, -\frac{d}{3}H^2, 0\right)=d\ch(I_\ell).$$
The Gieseker semistability of $E$ implies that $\Hom(\cO_X(kH), E) = 0$ for $k \geq 0$. We also have $\chi(\cO_X(kH), E) = 0$ for $k =0, 1$; thus the definition of Ulrich bundles implies that $\Hom(\cO_X(kH), E[3]) = 0$. Therefore, $E \in \ku(X)$. The first step to prove Theorem~\ref{thm_ulrich} is to show that Ulrich bundles are semistable objects in the Kuznetsov component with respect to the stability conditions $\sigma(\alpha,\beta)$ defined in Theorem~\ref{thm_stabcondinduced}. 

\begin{Prop}
\label{prop_stability_Ulrichbdl}
Every $E \in \mathfrak{M}^{U}_{d}$ is $\sigma(\alpha, \beta)$-semistable in $\ku(X)$ for any $(\alpha,\beta) \in V$, with $V$ as in \eqref{setV}. 
\end{Prop}

\begin{proof}
Since $\mu_H(E)=0$ and $E$ is Gieseker semistable, we get $E \in \Coh^{\beta}(X)$ for $\beta < 0$, and by~\cite[Lemma 2.7]{bayer:the-space-of-stability-conditions-on-abelian-threefolds}, the Ulrich bundle $E$ is $\sigma_{\alpha,\beta}$-semistable for $\alpha \gg 0$. Theorem 3.1 of \cite{feyz:slope-stability-of-restriction} implies that $E$ is $\sigma_{\alpha, \beta_0}$-semistable for any $\alpha >0$ and 
\begin{equation*}
    \beta_0 = - \frac{1}{3d(d-1)}. 
\end{equation*}
Since $E$ is locally-free, Proposition~\ref{prop-rotated-stability} implies that $E$ is $\sigma^0_{\alpha, \beta_0}$-semistable; thus by definition, $E$ is  $\sigma(\alpha,\beta_0)$-semistable for $\alpha$ small enough because $E \in \ku(X)$. Hence the claim follows from Corollary~\ref{cor-s-invariant-ku}. 
\end{proof} 

We denote by $M^{\sigma(\alpha, \beta)}_{\ku(X)}( d[I_{\ell}])$ the moduli space parametrising $\sigma(\alpha, \beta)$-semistable objects in $\ku(X)$ of class $d[I_{\ell}]$. Proposition~\ref{prop_stability_Ulrichbdl} implies that we have an embedding
\begin{equation}\label{embed}
    \mathfrak{M}^{U}_{d} \hookrightarrow M^{\sigma(\alpha, \beta)}_{\ku(X)}( d[I_{\ell}]).
\end{equation}
The next step is to show that \eqref{embed} is an open embedding. By Corollary~\ref{cor-s-invariant-ku}, we may assume the pairs $(\alpha, \beta) \in V$, with $V$ as in \eqref{setV}, are on the curve $\alpha^2 =\beta^2 -2/3$, so (up to a shift) any $E \in M^{\sigma(\alpha, \beta)}_{\ku(X)}( d[I_{\ell}])$ has maximum phase $1$ in the heart $\cA(\alpha, \beta)$, as in \eqref{eq_Zofline=-1}. Therefore, $E \in \Coh^0_{\alpha,\beta}(X)$ is $\sigma^0_{\alpha, \beta}$-semistable of phase $1$. We claim $E$ is $\sigma_{\alpha, \beta}$-semistable; otherwise, Proposition~\ref{prop-rotated-stability} implies that $E$ lies in the exact triangle 
\begin{equation}\label{triangle}
F[1] \rightarrow E \rightarrow T, 
\end{equation}
where $F \in \mathcal{F}_{\alpha, \beta}$ and $T \in \Coh_0(X)$. By our choice of $(\alpha, \beta)$, we know $\mu_{\alpha, \beta}(F) =0$. Since $F \in \mathcal{F}_{\alpha, \beta}$, we know 
\begin{equation*}
    \mu^+_{\alpha, \beta}(F) \leq 0 = \mu_{\alpha, \beta}(F).  
\end{equation*}
Thus $F$ is $\sigma_{\alpha, \beta}$-semistable of slope zero. But we know $\mu_{\alpha, \beta}(\cO_X(-2H)[1]) < 0$; thus 
\begin{equation*}
    \hom(\cO_X, F[2]) = \hom(F, \cO_X(-2H)[1]) =0. 
\end{equation*}
Since $E \in \ku(X)$, we know $\hom(\cO_X, E) = 0$; thus the exact triangle \eqref{triangle} implies that $\hom(\cO_X, T) = 0$, which is not possible as $T$ is a skyscraper sheaf, and so $E$ is $\sigma_{\alpha, \beta}$-semistable. 

Hence, up to a shift, we may assume any $E \in M^{\sigma(\alpha, \beta)}_{\ku(X)}( d[I_{\ell}])$ lies in the heart $\Coh^{\beta}(X)$. Thus $\cH^{i}(E) = 0$ if $i \neq 0, -1$, and if $\cH^{-1}(E) \neq 0$, then it is a torsion-free sheaf. So if we consider the locus of objects $E \in M^{\sigma(\alpha, \beta)}_{\ku(X)}( d[I_{\ell}])$ with rank$(\cH^{-1}(E)) = 0$, then we get precisely the locus of sheaves which we know is non-empty and by~\cite[Theorem 7.7.5]{grothendieck} is an open sublocus. Moreover, being Ulrich is an open property, see \cite{ulrich}, so $\mathfrak{M}^{U}_{d}$ is an open subset of $M^{\sigma(\alpha, \beta)}_{\ku(X)}( d[I_{\ell}])$, as claimed. 

By Theorem~\ref{thm_stabcondinduced}, we have
$$
M^{\sigma(\alpha, \beta)}_{\ku(X)}( d[I_{\ell}]) \cong M^{\sigma(\alpha, \beta)}_{\ku(X)}( d[\cS(I_{\ell})]),
$$
and Corollary~\ref{cor_barsigmavssigmaalphabeta} gives the isomorphism 
\begin{equation}\label{iso}
   M^{\sigma(\alpha, \beta)}_{\ku(X)}( d[\cS(I_{\ell})]) \cong M^{\overline{\sigma}(b, w)}_{\ku(\PP^2, \cB_0)}(2d[\cB_0] -d[\cB_{-1}]),  
\end{equation}
where the stability conditions $\overline{\sigma}(b, w)$ are defined in Proposition~\ref{prop.2-dim}.  

\begin{Prop} \label{prop_stableB_0modinMd}
Take an object $E \in \Db(\PP^2, \cB_0)$ of class $2d[\cB_0] -d[\cB_{-1}] = d\left(4, -3, \frac{3}{2}\right)$. Then $E$ is a shift of a Gieseker-$($semi$\,)$stable sheaf if and only if $E$ lies in $\ku(\PP^2, \cB_0)$ and is $\overline{\sigma}(b, w)$-$($semi$\,)$stable for some $(b,w) \in U$, where $-\frac{5}{4} \leq b < -\frac{3}{4}$.    
\end{Prop}

\begin{proof}
First assume $E$ is a Gieseker-(semi)stable sheaf. Then its Gieseker-slope is less than that of $\cB_1$, so $ \hom(\cB_1, E) =0$. Moreover,  
\begin{equation*}
    \hom(\cB_1, E[2]) = \hom(E, \cB_0) = 0;   
\end{equation*}
thus $\hom(\cB_1, E[1]) = -\chi(\cB_1, E) = -2d\chi(\cB_1, \cB_0) +d \chi(\cB_1, \cB_{-1}) = 0$. This shows that $E \in \ku(\PP^2, \cB_0)$. Since $E$ is Gieseker (semi)stable, it is $\overline{\sigma}_{b,w}$-(semi)stable for $b < -\frac{3}{4}$ and $w \gg 0$. Then~\cite[Remark 5.12]{bayer:stability-conditions-kuznetsov-component} implies that $E$ is $\overline{\sigma}(b, w)$-semistable. 

Conversely, take a $\overline{\sigma}(b, w)$-(semi)stable object $E$ in $\ku(\PP^2, \cB_0)$. We may assume $E \in \cA_b$ and by the uniqueness of $\cS_{\cB_0}$-invariant stability conditions, we can assume $w \gg 0$. Thus~\cite[Remark 5.12]{bayer:stability-conditions-kuznetsov-component} shows that $E$ is in $\Coh^b(\PP^2, \cB_0)$ and is $\overline{\sigma}_{b,w}$-(semi)stable. This immediately implies that $E$ is a Gieseker-(semi)stable sheaf.  
\end{proof}

Thus we get the isomorphism 
\begin{equation*}
M^{\overline{\sigma}(b, w)}_{\ku(\PP^2, \cB_0)}(2d[\cB_0] -d[\cB_{-1}]) \cong M_{(\PP^2,\ \cB_0)}^{ss}(2d[\cB_0] -d[\cB_1]) \eqqcolon \mathcal{M}_d,  
\end{equation*}
where $\mathcal{M}_d$ parametrises Gieseker-semistable sheaves in $\Coh(\PP^2, \cB_0)$ of class $2d[\cB_0] -d[\cB_1]$. The last step is to show the following. 

\begin{Prop} \label{prop_irreducible}
The moduli space $\mathcal{M}_d$ is irreducible. 
\end{Prop}

\begin{proof}
The argument is the same as Step 3 in the proof of~\cite[Theorem 2.12]{macri:acm-bundles-cubic-3fold}; we include it for completeness.  

The first claim is that $\mathcal{M}_d$ is connected. If $d= 1$, then \eqref{iso} gives 
$$
\mathcal{M}_1 \cong M^{\sigma(\alpha, \beta)}_{\ku(X)}([I_{\ell}]),
$$ 
and the latter is isomorphic to the Fano surface of lines in $X$ by~\cite[Theorem 1.1]{pertusi:some-remarks-fano-threefolds-index-two}. In particular, $\mathcal{M}_1$ is connected. For $d>1$, the strictly semistable locus is covered by the images of the natural maps $\varphi_{d_1, d_2} \colon \mathcal{M}_{d_1} \times \mathcal{M}_{d_2} \to \mathcal{M}_d$ for $1 \leq d_1 \leq  d_2 \leq d$ and $d_1+d_2=d$ which send any pair $(E_1, E_2)$ to $E_1 \oplus E_2$. Thus by induction, we deduce that the semistable locus of $\mathcal{M}_d$ is connected as well. Arguing exactly as in~\cite[Lemma 4.1]{lehn:singular-symplectic-moduli-space}, we obtain the non-existence of a connected component of $\mathcal{M}_d$ consisting of purely stable sheaves.
We conclude that $\mathcal{M}_d$ is connected.  

Finally, note that for all $E \in \mathcal{M}_d$, we have Ext$^2(E, E) = \Hom(E, E\otimes_{\cB_0} \cB_{-1})^{\vee} = 0$. This implies that $\mathcal{M}_d$ is normal, combining well-known results in deformation theory and Quot schemes (see~\cite[Theorem~2.12]{macri:acm-bundles-cubic-3fold}).
\end{proof}

Since $\mathfrak{M}^{U}_{d}$ is an open subset of the irreducible space $\mathcal{M}_d$,
it is irreducible. This ends the proof of Theorem~\ref{thm_ulrich}.  
\\
We also observe the following property, which is a consequence of the previous computations.

\begin{Cor}
\label{cor_moduliprojective}
The moduli space $M^{\sigma(\alpha, \beta)}_{\ku(X)}( d[I_{\ell}])$ is irreducible and projective.
\end{Cor}

\begin{proof}
By the above computation, we have $M^{\sigma(\alpha, \beta)}_{\ku(X)}( d[I_{\ell}]) \cong \mathcal{M}_d$. Since the latter is a moduli space of Gieseker-semistable sheaves, constructed as a GIT quotient of an open subset of a Quot-scheme, it is projective. Hence Proposition~\ref{prop_irreducible} implies the statement.
\end{proof}
In particular, by Theorem~\ref{thm_uniqueSinv}, we deduce that $M^{\sigma}_{\ku(X)}( d[I_{\ell}])$ is irreducible and projective for every Serre-invariant stability condition $\sigma$ on $\ku(X)$.

\newcommand{\etalchar}[1]{$^{#1}$}
\providecommand{\bysame}{\leavevmode\hbox to3em{\hrulefill}\thinspace}
\providecommand{\MR}{\relax\ifhmode\unskip\space\fi MR }
\providecommand{\MRhref}[2]{%
  \href{http://www.ams.org/mathscinet-getitem?mr=#1}{#2}
}
\providecommand{\href}[2]{#2}

\end{document}